\documentclass[reqno]{amsart}
\usepackage{amsmath,amsfonts,amssymb,amsthm}
\usepackage[usenames,dvipsnames]{xcolor}
\usepackage{graphicx}
\usepackage[hypertex]{hyperref} 
\usepackage{url}
\hypersetup{linkcolor=Violet,citecolor=PineGreen}
\newtheorem{teor}{Theorem}[section]

\newtheorem{prop}[teor]{Proposition}
\newtheorem{coro}[teor]{Corollary}
\theoremstyle{definition}
\newtheorem{defi}[teor]{Definition}

\newtheorem{hipos}[teor]{Hypotheses}
\newtheorem{nota}[teor]{Remark}

\newtheorem{notas}[teor]{Remarks}
\numberwithin{equation}{section}

\newcommand{\N}{\mathbb{N}}

\newcommand{\R}{\mathbb{R}}

\newcommand{\T}{\mathbb{T}}

\newcommand{\mA}{\mathcal{A}}
\newcommand{\mB}{\mathcal{B}}
\newcommand{\mC}{\mathcal{C}}

\newcommand{\mF}{\mathcal{F}}

\newcommand{\mH}{\mathcal{H}}

\newcommand{\mK}{\mathcal{K}}

\newcommand{\mM}{\mathcal{M}}
\newcommand{\mO}{\mathcal{O}}
\newcommand{\mP}{\mathcal{P}}

\newcommand{\mR}{\mathcal{R}}
\newcommand{\mS}{\mathcal{S}}
\newcommand{\mU}{\mathcal{U}}

\newcommand{\ep}{\varepsilon}
\newcommand{\pu}{{\cdot}}

\newcommand{\W}{\Omega}
\newcommand{\WW}{W^{1,\infty}}
\newcommand{\w}{\omega}
\newcommand{\wt}{\w\pu t}
\newcommand{\ws}{\w\pu s}
\newcommand{\hw}{\wih\w}
\newcommand{\hwt}{\wih\w\pu t}

\newcommand{\wit}{\widetilde}
\newcommand{\wih}{\widehat}

\newcommand{\n}[1]{\|#1\|}
\newcommand{\Frac}[2]{\displaystyle\frac{#1}{#2}}

\newcommand{\lsm}{\left[\begin{smallmatrix}}
\newcommand{\rsm}{\end{smallmatrix}\right]}

\newcommand{\Lin}{\text{\rm Lin}}
\newcommand{\dist}{\text{\rm dist}}

\begin{document}
\title[Global attractor for a state-dependent neural network]
{Existence of global attractor for a nonautonomous state-dependent
delay differential equation of neuronal type}
\author[C.~Elia]{Cinzia Elia}
\author[I.~Maroto]{Ismael Maroto}
\author[C.~N\'{u}\~{n}ez]{Carmen N\'{u}\~{n}ez}
\author[R.~Obaya]{Rafael Obaya}
\address[I.~Maroto, C.~N\'{u}\~{n}ez, R.~Obaya]
{Departamento de Matem\'{a}tica Aplicada, Universidad de
Valladolid, Paseo del Cauce 59, 47011 Valladolid, Spain}
\address[C.~Elia]{Dipartimento di Matematica,
Universit\`{a} degli Studi di Bari, Via Orabona 4,
70125 Bari, Italy}
\email[Cinzia Elia]{cinzia.elia@uniba.it}
\email[Ismael Maroto]{ismmar@eii.uva.es}
\email[Carmen N\'{u}\~{n}ez]{carnun@wmatem.eis.uva.es}
\email[Rafael Obaya]{rafoba@wmatem.eis.uva.es}
\thanks{Partly supported by Ministerio de Econom\'{\i}a y Competitividad / FEDER
under project MTM2015-66330-P, by Ministerio de Ciencia, Innovaci\'{o}n y Universidades
under project RTI2018-096523-B-I00, and by European Commission under project
H2020-MSCA-ITN-2014.}
\date{}
\begin{abstract}
The analysis of the long-term behavior of the mathematical model of a neural
network constitutes a suitable framework to develop new tools for the dynamical
description of nonautonomous state-dependent delay equations (SDDEs).
The concept of global
attractor is given, and some results which establish properties ensuring
its existence and providing a description of its shape, are proved.
Conditions for the exponential stability of the global attractor
are also studied. Some properties
of comparison of solutions constitute a key in
the proof of the main results, introducing methods of monotonicity
in the dynamical analysis of nonautonomous SDDEs.
Numerical simulations of some illustrative models show
the applicability of the theory.
\end{abstract}
\maketitle
\section{Introduction}\label{sec1}
The analysis of nonautonomous differential equations constitutes a complex field
in Mathematics on which many researchers, starting from Poincar\'{e}, have actively
worked.~A large quantity of problems are relevant not only because of
their theoretical interest, but also due to their fundamental
role in the more accurate mathematical modeling of many different actual phenomena.
Generally speaking, the goal is to
understand the way in which the intrinsic dynamics of the dynamical system
determined by a nonautonomous equation, which is
due to the explicit time dependence of the law, affects the behavior of the phenomenon
under analysis. Quite often the dynamical scenario described by the analysis
reproduces well-known patterns of the autonomous case; but sometimes there appear
new scenarios which cannot occur in the autonomous or periodic cases, with a
high degree of complexity.
\par
Certain properties of regularity on the time variation of the functions determining
the equation allow us to include it in a collective family of equations of the
same type, whose solutions define a continuous or random flow or semiflow of
skew-product type. This procedure has been the initial point to develop a
wide collection of dynamical techniques, both analytical and numerical,
which constitute the core of a robust theory. Suitable references can be
those of Sell \cite{sell},
Sacker and Sell \cite{SackerSell},
Chow and Leiva \cite{chle1,chle2},
Arnold \cite{arnol},
Shen and Yi \cite{shyi},
Cheban {\em et al} \cite{chks},
Kloeden and Rassmussen \cite{klra},
Carvalho {\em et al} \cite{chlr},
Caraballo and Han \cite{caha},
Johnson {\em et al} \cite{jonnf},
as well as the works cited therein.
\par
The main objective of this paper is to provide new tools for the
dynamical description of nonautonomous functional differential
equations with state-dependent delay (SDDEs for short), focusing in
models of neural networks. We will show that some methods which
have been frequently used in the analysis of the long term dynamics
of neural networks with time-dependent delay can be adapted to the
case of state-dependent delay. The weak regularity of
the solutions of the SDDEs with respect to the initial states makes this
extension not trivial.
\par
Our tools are, roughly speaking, two. The first one is to define
and describe global and pullback attractors in this setting, as well
as to establish criteria ensuring their existence and providing a global
description of their shapes. As far as we know, this is the first time that
a theory on the
existence and dynamical properties of attractor sets is given
in the setting of nonautonomous SDDEs.
As second tool, we introduce arguments of the theory of nonautonomous
monotone dynamical  systems in the study of nonautonomous SDDEs.
We compare our SDDEs with simpler types of nonautonomous ordinary differentail equations
which satisfy a quasimonotone condition to determine the area containing the global attractor.
In a similar way, we compare the linearized families of SDDEs equations with families
of functional equation for which the delay is just time-dependent, in order to obtain appropriate
bounds for the upper Lyapuov exponents of minimal sets. In same cases these bounds show that
the attractor has a simple shape and the existence of globally exponentially stable recurrent solutions.
To our knowledge this is the first time that methods of nonautonomous monotone dynamical systems
are applied in the context of SDDEs.
\par
A large number of researchers have been interested in the
dynamics induced by SDDEs, motivated both by its high theoretical
interest and by the increasing number of models of applied sciences
which respond to this pattern. Among them, we can mention
Hartung~\cite{hart1, hart6},
Wu~\cite{wu},
Hartung {\em et al.}~\cite{hkvv},
Mallet-Paret and Nussbaum \cite{mpna},
Hu and Wu \cite{huwu},
Hu {\em et al.} \cite{huvz},
Walther \cite{walth0,walth},
He and de la Llave \cite{hell, hell2},
Krisztin and Rezounenko \cite{krre}
and Maroto {\em et al.}~\cite{mano1,mano2},
as well as the many references therein.
In particular, the regularity properties of the
solutions of families of SDDEs which we will use in this paper
are described in \cite{mano1,mano2}.
\par
Next we describe briefly the contents of this
paper, which are organized in three sections.
\par
We begin Section \ref{sec2} by recalling some standard notions of topological dynamics,
as well as the concepts of global attractor, exponentially stable set
and upper Lyapunov exponent. The last three definitions are referred to a
skew-product semiflow $(\W\times X,\zeta\,,\R^+)$ projecting on a flow
$(\W,\sigma,\R)$, where $\W$ is a compact metric space and $X$ is a Banach space.
For the main
purposes of this paper, $\W$ will be the hull of the almost periodic coefficients
of the SDDEs which models our neural network and $\sigma$
will be the time-translation flow on $\W$ (see Section \ref{3.sechull}
for the details). This hull procedure will allow us to construct
a family of SDDEs of the form
\begin{equation}\label{1.eq}
 \dot{y}(t)=F(\wt, y(t), y(t-\tau(\wt, y_{t})))\,,\qquad t\geq 0
\end{equation}
for $\w\in\W$, where $\{\wt\,|\;t\in\R\}$ is the $\sigma$-orbit of the point
$\w\in\W$. (The standard regularity conditions assumed on the map $F$ and
the delay $\tau$ are described in detail in Section \ref{sec2}).
The space $X$ will be $\WW_2:=\WW([-r,0],\R^2)$, where $r>0$ is the maximum delay.
And the (local) semiflow $\zeta\,$ will be given by
\[
 \zeta\colon\mU_2\subseteq\R^+\!\!\times\W\times\WW_2
 \to\W\times\WW_2\,,\quad (t,\w,x)\mapsto(\wt,u(t,\w,x))\,,
\]
with $u(t,\w,x)(s):=y(t+s,\w,x)$ for $s\in[-r,0]$, where $y(t,\w,x)$ is
the solution of the equation corresponding to $\w\in\W$ and $x\in\WW_2$ with
$y(s,\w,x)=x(s)$ for $s\in[-r,0]$. The map $\zeta\,$
may be noncontinuous on $[0,r]\times\W\times\WW_2$.
However, it satisfies enough continuity properties to be still
a valuable tool in the long-term analysis to be carried out.
As a matter of fact, it has a continuous restriction to the
compatibility set $\mC^0_n:=\{(\w,x)\in\W\times C^1_n\,|\;
\dot{x}(0^-)=F(\w,x(0),x(-\tau(\w,x)))\}$.
This property and some others are summarized in Section \ref{sec2}.
The set $\mC_n^0$ is closed and invariant, but in general not
a differentiable manifold in this nonautonomous setting.
One more result, concerning the monotonicity properties
of $\zeta\,$ as well as comparison results for \lq\lq ordered\rq\rq~families
of SDDEs satisfying a quasi-monotonicity condition, completes
the section. It adapts an already classical
result of Smith \cite{smit}.
\par
Section \ref{sec3} contains the core results of this paper.
It begins with the description of the model for the biological
network, given by a two-dimensional system of nonautonomous
SDDEs with almost periodic coefficients.
It is a simplified model for two groups of neurons: the
internal action of each group is assumed to be instantaneous,
while the action onto the other group is assumed to be delayed,
and with state-dependent delay. As said above, the hull construction
includes this system in a family of the type \eqref{1.eq}.
Therefore, we can define a semiflow, part of whose orbits
are defined exactly by the solutions of the initial system.
The topological dynamics techniques allow us to show the existence
of a global attractor $\mA$, which is determinant to understand the
long term dynamics. It is important to emphasize that the
classical theory of attractors is carried out in the autonomous case,
while here we are dealing with a nonautonomous (and state-dependent) problem:
the fundamental tool to make this extension possible is
the skew-product formalism. It is also remarkable that the
problems that the absence of global continuity causes,
both in the definitions and in the proofs of the results,
can be solved thanks to the actual continuity properties.
\par
Section~\ref{sec3} also contains a brief explanation of the way in which
the existence and properties of the global attractor ensure the
existence and some properties of the so-called pullback attractor
(see e.g.~\cite{chlr}) for the process
defined from the initial nonautonomous SDDE.
In addition, under the additional hypotheses of the exponential
stability of the minimal sets of $\W\times\WW_2$, the attractor $\mA$
turns out to agree with the graph of a continuous function
$a\colon\W\to\WW_2$, which we call copy of the base.
Moreover, all the semiorbits are exponentially
attracted to $\mA$. It is proved in \cite{mano2} that the
exponential stability of a minimal set is equivalent to the
negativeness of its upper Lyapunov exponent.
A procedure which makes it easier to determine if this condition
holds completes Section~\ref{sec3}.
\par
In Section \ref{sec4} we carry out some numerical experiments
for a particular model. First, using the comparison methods previously described,
we delimit a region containing the global attractor.
We perform simulations of the model under study both in the forward
sense and in the pullback sense. The computer simulations suggest the
existence of an attractor for the numerical method and show that
the bounds we obtained for the containing area are quite accurate.
We also check that the conditions ensuring that the
global attractor is a copy of the base are fulfilled if we are
more exigent in the choice of the delay, and give numerical evidence
that the attractor we see in the simulations is indeed a copy of the base.
\par
We finally point out that the ideas here developed
shall be useful in the analysis of many other phenomena which can be
modeled by nonautonomous SDDEs.
\section{Basic notions and properties}\label{sec2}
The basic notions and some classical results on topological dynamics
required in the paper are recalled in this section,
whose contents may be found in Sell~\cite{sell},
Sacker and Sell~\cite{SackerSell, SackerSellDich}, Hale~\cite{hale4},
Chow and Leiva \cite{chle1,chle2}, Shen and Yi \cite{shyi},
and references therein.
\par
Let $\W$ be a complete metric space with distance $d_\W$. A {\em flow\/}
$(\W, \sigma, \R)$ is defined by a Borel measurable map
$\sigma\colon\R\times\W\rightarrow\W$, $(t,\w)\mapsto\sigma(t,\w)$
satisfying
\[
 \!\!\!\!\!\!\!\!\!\!
 \text{\hypertarget{(1)}(\text{f1}) $~\sigma_0=\text{Id}\,,$\hspace{2cm}
 \hypertarget{(2)}(\text{f2}) $~\sigma_{t+l}=\sigma_{t}\circ\sigma_{l}\;$
 for all $s,t\in\R$}\,,
\]
where $\sigma_t(\w):=\sigma(t,\w)$ for all $t\in\R$ and $\w\in\W$.
The flow is {\em continuous\/} if $\sigma$ is continuous.
The sets
$\{\sigma_t(\w)\,|\,t\in\R\}$,
$\{\sigma_t(\w)\,|\,t\ge 0\}$
and $\{\sigma_t(\w)\,|\,t\le 0\}$
are respectively the {\em orbit\/}, {\em positive\/} or {\em forward semiorbit}
and {\em negative\/} or {\em backward semiorbit} of the point $\w\in\W$.
If the forward (or backward) semiorbit is relatively compact,
the {\em omega-limit set} (resp.~{\em alpha-limit set})
{\em of the point $\w\in\W$} (or {\em of its semiorbit})
is the set of limits of sequences of the form
$(\sigma_{t_n}(\w))$ with $(t_n)\uparrow\infty$
(resp.~$(t_n)\downarrow\infty$).
A Borel set $\mM\subseteq\W$ is $\sigma$-{\em invariant\/}
(or just {\em invariant}, if no confusion arises)
if $\sigma_t(\mM)=\mM$ for all $t\in\R$, and it is
$\sigma$-{\em minimal\/} (or {\em minimal})
if it is compact, $\sigma$-invariant,
and it contains properly no nonempty compact $\sigma$-invariant subset.
Zorn's lemma ensures
that every $\sigma$-invariant compact set contains a minimal subset; and
clearly a compact $\sigma$-invariant subset is minimal if and only
if each one of its semiorbits is dense in it.
A flow $(\W, \sigma, \R)$ is {\em recurrent\/} or {\em minimal\/}
if $\W$ itself is minimal.
A continuous flow $(\W, \sigma, \R)$ is {\em almost periodic\/} if for any
$\ep>0$ there is a $\delta=\delta(\ep)>0$ such that, if $\w_1,\w_2\in\W$ satisfy
$d_\W(\w_1,\w_2)<\delta$, then
$d_\W(\sigma_t(\w_1), \sigma_t(\w_2))<\ep$ for all $t\in\R$.
The flow is {\em local\/}
if the map $\sigma$ is defined and satisfies \hyperlink{(1)}{(f1)} and
\hyperlink{(2)}{(f2)} on an open subset
$\mU\subseteq\R\times\W$ containing $\{0\}\times\W$.
\par
As usual, we represent $\R^\pm:=\{t\in\R\,|\,\pm t\ge 0\}$.
A {\em semiflow\/} $(\W,\sigma,\R^+)$ is given by
a Borel measurable map $\sigma\colon\R^+\!\!\times\W\rightarrow\W$,
$(t,\w)\mapsto\sigma(t,\w)$
satisfying \hyperlink{(1)}{(f1)} and \hyperlink{(2)}{(f2)}
for all $t,s\in\R^+$; and
it is continuous if $\sigma$ is continuous.
Positive semiorbits and omega-limit sets are defined as above.
A Borel subset $\mM\subseteq\W$ is {\em positively $\sigma$-invariant\/}
if $\sigma_t(\mM)\subseteq \mM$
for all $t\ge 0$.
A positively $\sigma$-invariant compact set $\mM$ is $\sigma$-{\em minimal\/}
(or {\em minimal}) if it
does not contain properly any positively $\sigma$-invariant
compact set. If $\W$ is minimal,
we say that the semiflow $(\W,\sigma,\R^+)$ is {\em minimal}.
The semiflow is {\em local\/} if the map $\sigma$
is defined, continuous, and satisfies \hyperlink{(1)}{(f1)} and
\hyperlink{(2)}{(f2)} on an open subset $\mU\subseteq\R^+\!\!\times\W$
containing $\{0\}\times\W$. In this case, the definitions of positively invariant set and
minimal set are the same as above. In particular, they are composed of globally defined
positive semiorbits, so that the restriction of the semiflow to one of these sets is global.
\par
Let the semiflow $(\W,\sigma,\R^+)$ be continuous. A point $\w\in\W$
{\em has a complete orbit in $\W$} if there exists a continuous map
$\theta_{\w}\colon\R\!\to\W$ such that $\theta_{\w}(0)=\w$
and $\sigma(t, \theta_{\w}(s))=\theta_{\w}(t+s)$ whenever $s\in\R$ and
$t\ge 0$. If the corresponding
negative semiorbit $\{\theta_\w(t)\,|\;t\le 0\}$ is relatively compact, then
it has an alpha-limit set, defined as above.
A set $\mM\subseteq\W$ is $\sigma$-invariant
if $\sigma_t(\mM)=\mM$ for all $t\ge 0$.
Note that this condition is quite
stronger than the positively $\sigma$-invariance.
It is not hard to prove that $\mM$ is
$\sigma$-invariant if and only if it is
composed by complete orbits of its elements: see, e.g., Lemma 1.4 of
Carvalho {\em et al.}~\cite{chlr}. In addition, a minimal set
is $\sigma$-invariant, as easily deduced from the minimality itself.
The same happens with the omega-limit sets of globally defined and
relatively compact semiorbits: see Proposition II.2.1 of~\cite{shyi}.
If $\mM$ is a $\zeta\,$-invariant set, the restricted semiflow
$(\mM,\sigma,\R^+)$ {\em admits a continuous
flow extension\/} if there exists a continuous
flow $(\mM,\bar{\sigma},\R)$ such that $\bar{\sigma}(t,\w)=\sigma(t,\w)$
for all $t\in\R^+$ and $\w\in\mM$.
If $\mM$ is locally compact, then the existence of a continuous flow
extension is equivalent to the uniqueness of
the complete orbit in $\mM$ of
each one of its points: see Theorem II.2.3 of~\cite{shyi}.
\par
Now let $(\W,\sigma,\R^+)$ be a global {\em continuous\/}
semiflow on a {\em compact\/} metric space $\W$, and let $X$ be a Banach space.
We will represent $\wt:=\sigma(t,\w)$. A local
semiflow ($\W\times X,\zeta\,,\R^+$)
is a local {\em skew-product semiflow with base\/} $(\W,\sigma,\R)$
{\em and fiber $X$} if it takes the form
\begin{equation}\label{2.defPi}
 \zeta\colon\mU\subseteq\R^+\!\!\times\W\times X\to\W\times X\,,\quad
 (t,\w,x)\mapsto(\wt, u(t,\w,x))\,.
\end{equation}
It is frequently assumed that the base semiflow is in fact a flow.
In this case, a compact set $\mK\subset\W\times X$ is a {\em copy of the base\/}
if it is $\zeta\,$-invariant and agrees with the graph of a continuous function
$k\colon\W\to X$. Note that, in this case, $u(t,\w,k(\w))=k(\wt)$ for all
$\w\in\W$ and all $t\ge 0$.
\par
The following definitions and properties refer to the case of
a {\em continuous\/} skew-product semiflow $(\W\times X,\zeta\,,\R^+)$.
Recall that $d_\W$ is the distance in the metric space $\W$ and let $\n{\pu}_X$
be the norm in the Banach space $X$. Given two subsets $\mC_1$ and $\mC_2$ of
$\W\times X$, we denote the {\em Hausdorff semidistance\/}
from $\mC_1$ to $\mC_2$ by
\begin{equation}\label{2.dist}
 \dist(\mC_1,\mC_2):=\sup_{(\w_1,x_1)\in\mC_1}\left(\inf_{(\w_2,x_2)\in\mC_2}
 \big(d_\W(\w_1,\w_2)+\n{x_1-x_2}_X\big)\right).
\end{equation}
For further purposes we recall that the {\em Hausdorff distance}
is defined by
\begin{equation}\label{2.distreal}
 d_\mH(\mC_1,\mC_2):=\max\big(\dist(\mC_1,\mC_2),\dist(\mC_2,\mC_1)\big)\,,
\end{equation}
and that these definitions are valid if we substitute
$\W\times X$ by any metric space.
\begin{defi}\label{2.defatt}
A set $\mS\subset\W\times X$ {\em attracts a set
$\mC\subseteq\W\times X$ under $\zeta\,$} if
$\zeta_{\,t}(\mC)$ is defined for all $t\ge 0$ and
$\lim_{t\to\infty}\dist(\zeta_{\,t}(\mC),\mS)=0$. The semiflow
$\zeta\,$ is {\em bounded
dissipative\/} if there exists a bounded set $\mS$ attracting all
the bounded subsets of $\W\times X$ under $\zeta\,$.
A set $\mA\subset\W\times X$
is a {\em global attractor\/} if it is compact, $\zeta\,$-invariant, and it
attracts every bounded subset of $\W\times X$ under $\zeta\,$.
Finally, a set $\mS$ is {\em absorbing under $\zeta\,$} if, for any
bounded set $\mB$, there exists $t_0=t_0(\mS,\mB)$ such that
$\zeta_{\,t}(\mB)\subseteq\mS$ for all $t\ge t_0$.
\end{defi}
\begin{notas}\label{2.notas}
1.~It is immediate to observe that a semiflow $\zeta\,$ needs to be
globally defined in order to be bounded
dissipative, and that the existence of a bounded absorbing set
ensures the bounded dissipativity of the semiflow
and the boundedness of any semiorbit.
\smallskip\par
2.~If a global attractor $\mA$ exists, then it contains any
other closed, bounded, and $\zeta\,$-invariant set $\mB$:
$0=\lim_{t\to\infty}\dist(\zeta_{\,t}(\mB),\mA)=\lim_{t\to\infty}\dist(\mB,\mA)$,
which ensures that $\mB\subseteq\mA$. In particular,
any $\zeta\,$-minimal set is contained in $\mA$.
A similar argument shows that $\mA$ is contained
in any closed
bounded set that attracts all the bounded subsets of $\W\times X$ under $\zeta\,$.
In particular, the attractor $\mA$ is unique, and it
is contained in any absorbing set.
\smallskip\par
3.~Recall that the $\zeta\,$-invariance of the
global attractor $\mA$ means that any of its elements has
a complete orbit in $\mA$. If fact, a point $(\w,x)$ belongs to $\mA$ if and only
it admits a complete orbit in $\W$ which is bounded:
see e.g.~Theorem 1.7 of \cite{chlr}.
\end{notas}
The next concept will be fundamental in the proofs of the main results.
\begin{defi}\label{2.defmonotone}
Suppose that the Banach space $X$ is partially ordered.
The semiflow $(\W\times X,\zeta\,,\R^+)$ defined by~\eqref{2.defPi}
is {\em monotone\/} if, for all $\w\in\W$ and all
$x_1,x_2\in X$ satisfying $x_1\le x_2$, it is $u(t,\w,x_1)\le u(t,\w,x_2)$ for all
$t\ge 0$ such that $(t,\w,x_1)$ and $(t,\w,x_2)$ belong to $\mU$.
In the case that the semiflow $\zeta\,$ is induced by a family of
differential equations, the elements of the family are {\em cooperative}.
\end{defi}
Let us now give the definition of uniform exponential stability, which refers to
a compact set $\mK\subset\W\times X$ projecting over the whole base; i.e.,
such that $\mK_\w:=\{x\in X\,|\;(\w,x)\in\mK\}$
is nonempty for all $\w\in\W$
(which is always the case if $\mK$ is positively $\zeta\,$-invariant and $\W$
is minimal).
\begin{defi}\label{2.defi3}
A positively $\zeta\,$-invariant compact set $\mK\subset\W\times X$ projecting over the whole
base is {\em exponentially stable} if there exist $\delta_0>0$, $k\ge 1$ and $\alpha>0$,
such that, if $(\w,\bar x)\in \mK$ and $(\w,x)\in\W\times X$
satisfy $\n{x-\bar x}_X<\delta_0$,
then $u(t,\w,x)$ is defined for $t\in[0,\infty)$ and
$\n{u(t,\w,x)-u(t,\w,\bar x)}_X\le k\,e^{-\alpha·t}\,\n{x-\bar x}_X$ for all $t\ge 0$.
The restricted semiflow $(\mK, \zeta\,, \R^+)$ is said to be {\em exponentially stable}.
\end{defi}
The last definition
refers to the case of a {\em linear\/} skew-product semiflow.
A global continuous skew-product semiflow $\zeta\,$ is {\em linear\/}
if it takes the form
\[
 \zeta\colon\R^+\!\!\times\W\times X\to\W\times X\,,\quad
 (t,\w,x)\mapsto(\wt, \phi(t,\w)\,x)\,,
\]
where $\phi(t,\w)\colon X\to X$ is a bounded linear operator.
(See Remark 2.5 of \cite{mano2} in order to see that the next definition
makes sense also in the case that the base flow $(\W,\sigma,\R^+)$
is a semiflow and not a flow.)
\begin{defi}\label{2.defLyap}
The {\em upper Lyapunov exponent of the set $\W$ for the semiflow $(\W,\zeta\,,\R^+)$} is
\[
 \lambda_{\W}:=\sup_{\w\in\W}\left(
 \sup_{x\in X,\,x\ne 0}\lambda^+_s(\w,x)\right)\,,
\]
where
\begin{equation}\label{2.lyapind}
 \lambda^+_s(\w,x):=\limsup_{t\to\infty}\frac1{t}\:\ln\n{\phi(t,\w)\,x}_X\,.
\end{equation}
\end{defi}
Some notation used throughout the paper is now described.
Given two Banach spaces $X$ and $Y$ with norms $\n{\pu }_X$ and $\n{\pu }_Y$,
$\Lin(X,Y)$ represents the set of bounded linear maps $\phi\colon X\to Y$ equipped
with the operator norm $\n{\phi}_{\Lin(X,Y)}=\sup_{\n{x}_X=1}\n{\phi(x)}_Y$.
Let us fix $r>0$. The set $C_n$ is the Banach space of continuous functions
$C([-r, 0], \R^n)$ equipped with the norm $\n{\psi}_{C_n}:=\sup_{s\in[-r, 0]}|\psi(s)|$,
where $|\cdot|$ represents the Euclidean norm in $\R^n$.
The set $L^\infty$ is
the space of Lebesgue-measurable functions $\psi\colon[-r, 0]\rightarrow\R^n$
which are {\em essentially bounded\/}, which means that there exists $k\ge 0$ such
that the set $\{x\in[-r, 0]\:|\;|\psi(x)|>k\}$ has zero measure. The norm on $L^\infty$
is defined as the inferior of the set of real numbers $k\ge 0$ with the previous
property and denoted by $\n\pu _{L^\infty}$.
The set $\WW_n$ is the Banach space of
Lipschitz-continuous functions $\psi\colon[-r, 0]\rightarrow\R^n$ equipped with the
norm
\[
 \n{\psi}_{\WW_n}:=\max\{\n{\psi}_{C_n}, \n{\dot{\psi}}_{L^\infty}\}\,.
\]
The subset of $\WW_n$ of the $C^1$-functions on $[-r,0]$
will be denoted by $C^1_n$.
Finally, given a continuous function
$x\colon[-r, \gamma]\to\R^{n}$ for $\gamma>0$ and a time
$t\in[0,\gamma]$, we denote by $x_t\in C_n$ the function defined by
$x_t(s):=x(t+s)$ for $s\in[-r, 0]$.
\subsection{Some basic facts on state-dependent delay equations}\label{sec5}
Let $(\W,\sigma,\R)$ be a continuous flow on a compact metric space, and
let us consider the family of nonautonomous SDDEs
\begin{equation}\label{2.eq}
\dot{y}(t)=F(\wt, y(t), y(t-\tau(\wt, y_{t})))\,,\qquad t\geq 0\,,
\end{equation}
for $\w\in\W$, where $F$ and $\tau$ satisfy the following conditions:
\begin{itemize}
  \item[\bf H1] \hypertarget{2.H1}
  $F\colon\W\times\R^{n}\times\R^{n}\to\R^{n}$ is continuous,
  and its partial derivatives with respect to the second and third arguments
  exist and are continuous on $\W\times\R^{n}\times\R^{n}$.
  \item[\bf H2] \hypertarget{2.H2}
  $\tau\colon\W\times C_n\to[0,r]$ is continuous and differentiable
  in the second argument, with $D_2\tau\colon\W\times C_n\to\Lin(C_n,\R)$ continuous.
\end{itemize}
We will use the notation \eqref{2.eq}$_\w$ to refer
to the system of this family corresponding to the point $\w$, and will proceed
in an analogous way for the rest of the equations appearing in the paper.
\par
The classical theory of finite-delay differential equations provides at least
a solution $x(t)$ of a functional differential equation
$x’(t) = g(t,x_t)$ whenever $g$ is continuous: see e.g. Chapter 2 of
\cite{hale2}. But the uniqueness requires additional conditions on the
Lipschitz behavior of $g$, which are not guaranteed by the conditions
{\rm \hyperlink{2.H1}{H1}} and {\rm \hyperlink{2.H2}{H2}}:
a simple adaptation to the case of finite delay of the
the example of \cite{wins} described in
Section 3.1 of \cite{hkvv} provides an equation with
two solutions for the same continuous initial data.
Theorem 1 of \cite{hart1} shows that the uniqueness is
indeed true under conditions {\rm \hyperlink{2.H1}{H1}}
and {\rm \hyperlink{2.H2}{H2}}
if the initial data is taken in $\WW_n$.
The next result, strongly based on Theorem 1 of \cite{hart1},
is proved in Theorem 3.3 and Corollary 3.4 of \cite{mano1} (which contain
more information). It provides a semiflow $(\W\times\WW_n,\zeta\,,\R^+)$ whose global
continuity cannot be ensured, but with strong continuity properties
which make it a valuable tool for the use of the techniques of
topological dynamics in the
analysis of the long-term behavior of the solutions of \eqref{2.eq}.
\begin{teor}\label{2.flujocontinuo}
Suppose that conditions {\rm \hyperlink{2.H1}{H1}} and
{\rm \hyperlink{2.H2}{H2}} hold. Then,
\begin{itemize}
\item[(i)] for $\w\in\W$ and $x\in\WW_n$, there exists a unique
 maximal solution $y(t,\w,x)$ of the equation \eqref{2.eq}$_\w$
 satisfying $y(s,\w,x)=x(s)$ for $s\in[-r,0]$, which is defined for
 $t\in[-r,\beta_{\w,x})$ with $0<\beta_{\w,x}\le\infty$. In particular,
 $y(t,\w,x)$ is continuous on $[-r,\beta_{\w,x})$ and satisfies~\eqref{2.eq}$_\w$
 on $(0,\beta_{\w,x})$, and there exists the
 lateral derivative $\dot y(0^+\!,\w,x)=F(\w,y(0),y(-\tau(\w,x))$.
\end{itemize}
Let us set
\begin{align*}
 &\mU_n:=\{(t,\w,x)\,|\;(\w,x)\in\W\times\WW_n,\;t\in[0,\beta_{\w,x})\}
 \subseteq\R^+\!\!\times\W\times\WW_n,\\
 &\wit\mU_n:=\{(t,\w,x)\in\mU_n\,|\;t\ge r\}\subset\R^+\!\times\W\times\WW_n,
\end{align*}
provide them with the subspace topology, and define
\begin{equation}\label{2.defu}
 u(t,\w,x)(s):=y(t+s,\w,x)
\end{equation}
for every $(t,\w,x)\in\mU_n$ and $s\in[-r,0]$, Then,
\begin{itemize}
\item[(ii)] $u(t,\w,x)\in\WW_n$ for all $t\in[0,\beta_{\w,x})$.
\item[(iii)] If $\sup_{t\in[0,\beta_{\w,x})}
 \n{u(t,\w,x)}_C<\infty$ then $\beta_{\w,x}=\infty$ and, in addition,
 the set $\{(\wt,u(t,\w,x))\,|\;t\in[r,\infty)\}
 \subset\W\times\WW_n$ is relatively compact.
\item[(iv)] The set $\mU_n$ is open in $\R^+\!\times\W\times\WW_n$ and the map
\begin{equation}\label{2.defPie}
 \zeta\colon\mU_n\subseteq\W\times\WW_n\to\W\times\WW_n\,,\quad (t,\w,x)\mapsto(\wt,u(t,\w,x))
\end{equation}
defines a semiflow.
\item[(v)] The map $\mU\to\W\times C_n\,,\;
 (t,\w,x)\mapsto(\wt,u(t,\w,x))$ is continuous.
\item[(vi)] The map
 $\wit\mU_n\to\W\times\WW_n,\;(t,\w,x)\mapsto(\wt,u(t,\w,x))$ is continuous.
\item[(vii)] Let us fix $\wit t\ge 0$ with $(\mU_n)_{\wit t}:=
 \{(\w,x)\,|\;(\wit t,\w,x)\in\mU\}$
 nonempty. Then the map
 $(\mU_n)_{\wit t}\to\W\times\WW_n,\;(\w,x)\mapsto
 (\w\pu\wit t,u(\wit t,\w,x))$ is continuous.
\item[(viii)] Let $\mK\subset\W\times\WW_n$
 be a positively $\zeta\,$-invariant compact set. Then the restriction
 of $\zeta\,$ to $\mK$ defines a global continuous semiflow on $\mK$.
\item[(ix)] The map $t\mapsto y(t,\w,x)$ is $C^1$ on $[-r,\beta_{\w,x})$ if
and only if $(\w,x)$ belongs to
\begin{equation}\label{3.defC0}
 \mC^0_n:=\{(\w,x)\in\W\times C^1_n\,|\;\dot{x}(0^-)=F(\w,x(0),x(-\tau(\w,x)))\}\,,
\end{equation}
which is closed and positively $\zeta\,$-invariant,
and $\zeta(t,\w,x)\in\mC^0_n$ if
$(t,\w,x)\in\mU_n$ and $t\ge r$. In addition,
if $\mU^{\,0}_n:=\{(t,\w,x)\in\mU_n\,|\;(\w,x)\in\mC^0_n\}$,
then
\begin{equation}\label{2.defPi0}
 \zeta^0\colon\mU^{\,0}_n\subseteq\R^+\times\mC^0_n\to\mC^0_n\,,\quad
 (t,\w,x)\mapsto(\wt,u(t,\w,x))
\end{equation} defines a continuous semiflow.
\end{itemize}
\end{teor}
\begin{notas}\label{2.notapseudo}
1.~The definition of $\zeta\,$-invariant set is
the same as in the case of a continuous semiflow.
The assertions in Theorem~\ref{2.flujocontinuo}(ix)
make it easy to prove that any omega-limit set (and
hence any minimal set) is $\zeta$-invariant, as in the continuous case.
Also the concepts given in Definitions \ref{2.defatt},
\ref{2.defmonotone} and \ref{2.defi3} can
be adapted to the case of our semiflow $(\W\times\WW_n,\zeta\,,\R^+)$,
and the properties stated in Remarks~\ref{2.notas} hold.
In fact, the properties of regularity
stated in the previous theorem will allow us to apply
standard topological methods to the map $\zeta$ on $\W\times\WW_n$.
\par
2.~Theorem~\ref{2.flujocontinuo}(ix) also ensures that
any $\zeta$-invariant set, as omega-limit
sets, minimal sets, and the global attractor (if it exists),
is contained in the set $\mC^0_n$ defined by \eqref{3.defC0},
for which the restricted semiflow is continuous.
And these sets are the key for the analysis of the long term dynamics.
\end{notas}
Let us now consider the usual componentwise order in $\R^n$:
$\lsm\alpha_1\\[-.2cm]\vdots\\\alpha_n\rsm\le\lsm\beta_1\\[-.2cm]\vdots\\\beta_n\rsm$
if and only if $\alpha_i\le\beta_i$ for $i=1,\ldots,n$.
It is easy to check that the Euclidean norm in $\R^n$ is
{\em monotone}: if $0\le \alpha\le\beta$, then $|\alpha|\le|\beta|$.
It follows easily from here
that a set $\mB\subset\R^n$ is bounded if and only if there exist $\alpha$ and $\beta$ in
$\R^n$ with $\alpha\le x\le\beta$ for all $x\in\mB$.
We also provide the Banach spaces $C_n$ and $\WW_n$ with the induced order
\begin{equation}\label{2.orderC}
  x\le y\qquad\text{if and only if}\qquad x(s)\le y(s)\quad\text{for all}\;s\in[-r,0]\,.
\end{equation}
The relation $\ge$
is defined on $\R^n$, $C_n$ and $\WW_n$ in the obvious way.
\par
\begin{defi}
Let $\tau\colon\W\times C_n\to[0,r]$ be a continuous function.
We say that a function $F\colon\W\times \R^n\times\R^n\to\R^n$ satisfies the {\em
quasi-monotonicity condition for the delay function $\tau$\/}
if it satisfies the following property:
\begin{itemize}
 \item[\bf Q] \hypertarget{2.QM}
 If $x^1,x^2\in \WW_n$
 satisfy $x^1\le x^2$ and $x^1_i(0)=x^2_i(0)$ for an
 $i\in\{1,\ldots,n\}$,
 then $F_i(\w,x^1(0),x^1(-\tau(\w,x^1)))\le F_i(\w,x^2(0),x^2(-\tau(\w,x^2)))$
 for all $\w\in\W$.
\end{itemize}
\end{defi}
The next comparison result adapts Theorem 1.1 of Chapter 5 of Smith \cite{smit}
to SDDEs of the type \eqref{2.eq}, and can be proved in the same way:
the required result on continuous dependence of the solutions with respect to parameters
can be proved as point (v) of Theorem 3.2 of \cite{mano1}. (To this
regard, see also Remark 3.5 of \cite{mano1}).
In the statement of Theorem \ref{2.teormonotone},
the notation $u_F(t,\w,x)$ corresponds to the
function defined by~\eqref{2.defu}, and $u_G(t,\w,x)$ corresponds to the analogous one
given by the new family $\dot y(t)=G(\wt, y(t), y(t-\tau(\wt, y_{t})))$ for $t\ge 0$ and
$\w\in\W$, with the same delay $\tau$ as the initial one, and with the same assumptions
on $G$ as on $F$.
\begin{teor}\label{2.teormonotone}
Let $\tau\colon\W\times C_n\to[0,r]$ satisfy {\rm \hyperlink{2.H2}{H2}},
and let $F,G\colon\W\times\R^n\!\times\R^n\to\R^n$ satisfy {\rm \hyperlink{2.H1}{H1}}.
Suppose also that either $F$ or $G$ satisfies {\rm \hyperlink{2.QM}{Q}} for $\tau$,
and that $F(\w,z^1,z^2)\le G(\w,z^1,z^2)$ for all $(\w,z^1,z^2)\in
\W\times\R^n\times\R^n$.
If $x$ and $\wit x$ in $\WW_n$ satisfy $x\le \wit x$, then
$u_F(t,\w,x)\le u_G(t,\w,\wit x)$ for all $t\ge 0$ in their common interval
of definition.
\end{teor}
An easy consequence follows (see Definition \ref{2.defmonotone} and
Remark~\ref{2.notapseudo}.1):
\begin{coro}\label{2.coromonotone}
Suppose that the family~\eqref{2.eq} satisfies conditions
{\rm \hyperlink{2.H1}{H1}} and {\rm \hyperlink{2.H2}{H2}}, and that
the function
$F$ satisfies {\rm \hyperlink{2.QM}{Q}} for the delay function $\tau$.
Then the local semiflow $(\W\times\WW_n,\zeta\,,\R^+)$ defined by~\eqref{2.defPie}
is monotone.
\end{coro}
\section{The long-term dynamics}\label{sec3}
We analyze in this section
the long-term dynamics of the solutions of a two-dimensional system of nonautonomous
SDDEs which models the so-called delayed cellular neural networks, namely
\begin{equation}\label{3.model0}
\!\!\!\!\left\{\!\!\begin{array}{l}
\dot{y_1}(t)\!=\!-a_1^0(t)\,y_1(t)\!+b_1^0(t)f(y_1(t))\!+c_1^0(t)\,g(y_2(t-\tau(y_1(t),
y_2(t))))\!+I_1^0(t)\,,\!\!\\[0.2cm]
\dot{y_2}(t)\!=\!-a_2^0(t)\,y_2(t)\!+b_2^0(t)f(y_1(t-\tau(y_1(t),
y_2(t))))\!+c_2^0(t)\,g(y_2(t))\!+I_2^0(t)\,.\!\!
\end{array}\right.
\end{equation}
This system describes the dynamics of two groups of neurons and the
relation of each group with itself, which we assume instantaneous,
and with the other, which we assume delayed with state-dependent delay:
the influence of the voltage of each neuron both in the synapse process and
in the signal transmission justifies the state-dependence of the delay in
the model.
The variables $y_1$ and $y_2$ describe an average value of the action potentials
of the neurons in each group. The coefficients $a_1^0$ and $a_2^0$
are the decaying terms,
the functions $b_1^0$ and $c_2^0$ are the synaptic coupling coefficients between different
neurons of a same group, and the functions $c_1^0$ and $b_2^0$ are the synaptic
coupling coefficients between neurons of different groups.
These coefficients take average values of the (positive or negative)
weights of the corresponding (excitatory or inhibitory) neuronal connections which,
due to the plasticity of the network, may vary
with respect to the time; and hence no assumption on their sign can be made.
The coefficients $I_1^0$ and $I_2^0$ denote nonautonomous external inputs, and
$f$ and $g$ are the bounded and increasing activation functions.
The delay is given by the function $\tau$ which, as mentioned before, depends
on the potentials (i.e. on the states) $y_1$ and $y_2$:
our model intends to be a more realistic approach
to a simple biological method than the classical ones,
described for instance in \cite{wu}.
\par
To state the precise conditions that we assume on this system, we recall
that a function $f\colon\R\to\R$ is {\em almost periodic\/} if
it is continuous and for any sequence $(s_n)$ of $\R$ there exist a
subsequence $(t_m)$ and a continuous function $f^*\colon\R\to\R$ such that
$(f_{t_m})$ converges to $f^*$ uniformly in $\R$, where $f_t(s):=f(t+s)$.
\begin{hipos}\label{3.hipos}
The coefficients $a_1^0$, $b_1^0$, $c_1^0$, $I_1^0$,  $a_2^0$, $b_2^0$,
$c_2^0$, and $I_2^0$ are almost periodic, with $a_1^0(t)\ge\delta$ and
$a_2^0(t)\ge\delta$ for all $t\in\R$ for a constant $\delta>0$.
In addition, the delay $\tau\colon\R^2\to[0, r]$ is a $C^1$ function,
and the {\em functions\/} $f$ and $g$ belong to $C^1(\R,[-1,1])$ and satisfy
$\dot{f},\,\dot g\colon\R\to[0,1]$.
\end{hipos}
The above is a representative simplified model of neuronal dynamics.
Our purpose is to provide a dynamical theory for nonautonomous SDDEs,
suitable to explain relevant features of the temporal evolution of
the neuronal activity. Higher dimensional systems and more general
expressions for the state-dependent delayed term can be considered
in the model: the corresponding theory may be developed in a similar way.
\subsection{The hull construction and the definition of the semiflow}\label{3.sechull}
We will include the system \eqref{3.model0} in a family of SDDEs such
that each one of its systems is given by the evaluation of a continuous
function along an orbit of a continuous flow
on a compact metric space. The reason to perform this procedure is that,
despite the nonautonomous character of our initial system, the solutions of
the whole family will allow us to define a semiflow and hence to apply techniques
from the topological dynamics. The conclusions are hence obtained for
all the systems of the family, and can be particularized {\em a posteriori\/}
for the initial system.
\par
The procedure to do this is the
classical Bebutov construction, which we explain now. We consider the function
$l_0\colon\R\rightarrow\R^{8}$ given by $l_0:=(a_1^0, b_1^0, c_1^0, I_1^0, a_2^0,
b_2^0, c_2^0, I_2^0)$. Let $\W$ be the {\em hull} of $l_0$, that is, the closure
in the compact-open topology of the set of almost periodic maps
$\{l_t\,|\,t\in\R\}$, with $l_t(s):=l_0(t+s)$ for $s\in\R$.
It is a classical result that $\W$ is a compact metric space,
and that the flow $\sigma$ defined on $\W$ by time-translation
(i.e., $\sigma\colon\R\times\W\rightarrow\W$, $(t,\w)\mapsto\wt$
with $\wt(s):=\w(t+s)$) is continuous: see e.g.~\cite{fink}.
The hypotheses made on the
coefficients of \eqref{3.model0} ensure that $l_0$ is an almost
periodic function, and thus $(\W,\sigma,\R)$
is a minimal almost periodic flow: see Chapter VI of~\cite{sell}.
We represent by $(a_1, b_1, c_1, I_1, a_2, b_2, c_2, I_2)$
the (continuous) zero-evaluation operator on $\W$, which maps each $\w\in\W$ to the
vector $\w(0)$ of $\R^8$. In this way, $a_1(\wt)=\w_1(t)$ 
(the fist component of $\w(t)\in\R^8$), and the same happens
with the remaining coefficients.
Once this is done, we can consider the family of SDDEs
\begin{equation}\label{3.model}
\left\{\!\!\begin{array}{l}
\dot{y_1}(t)=-a_1(\wt)\,y_1(t)\!+b_1(\wt)f(y_1(t))\\[.2cm]
\quad\;\qquad+\,c_1(\wt)\,g(y_2(t-\tau(y_1(t),
y_2(t))))\!+I_1(\wt)\,,\\[0.2cm]
\dot{y_2}(t)=-a_2(\wt)\,y_2(t)\!+b_2(\wt)f(y_1(t-\tau(y_1(t),
y_2(t))))\\[.2cm]
 \quad\;\qquad+\,c_2(\wt)\,g(y_2(t))\!+I_2(\wt)
\end{array}\right.
\end{equation}
for $\w\in\W$.
The initial system~\eqref{3.model0} belongs to this family: it agrees with
\eqref{3.model}$_{\w_0}$ for $\w_0:=l_0$. Note also that
$a_1(\w)\ge\delta$ and $a_2(\w)\ge\delta$ for all $\w\in\W$,
with $\delta$ provided by Hypotheses \ref{3.hipos}.
\par
It is easy to check that the family \eqref{3.model} satisfies conditions
{\rm \hyperlink{2.H1}{H1}} and {\rm \hyperlink{2.H2}{H2}} of Section~\ref{sec5}
in the two-dimensional case.
For each $(\w,x)\in\W\times\WW_2$,
we denote by $y(t,\w,x)$ the solution of the equation \eqref{3.model}$_\w$
with $y(s,\w,x)=x(s)$ for $s\in[-r,0]$,
which is defined on a maximal interval $[-r,\beta_{\w,x})$.
We also define $u(t,\w,x)(s):=y(t+s,\w,x)$ for $t\in[0,\beta_{\w,x})$ and
$s\in[-r,0]$, and consider the local semiflow
$(\W\times\WW_2,\zeta\,,\R^+)$, where, as in~\eqref{2.defPie},
\begin{equation}\label{3.skew}
 \zeta\colon\mU_2\subseteq\R^+\!\times\W\times\WW_2\to\W\times\WW_2\,,\quad
 (t,\w, x)\mapsto(\wt, u(t,\w,x))\,.
\end{equation}
It turns out that, in general, $\W$ is not a locally connected space:
see \cite{mese}. Hence, it cannot be identified with a differentiable manifold,
and this implies that $\mC^0_2$ cannot be provided with
the structure of a differentiable manifold, despite
the fact that the fiber of $\mC^0_2$ over each base point $\w$
is a continuously differentiable submanifold
of a Banach space varying with $\w$ (see Proposition 3.4 of \cite{walth}).
\subsection{Existence of the global attractor}
In the rest of this section, we work under Hypotheses \ref{3.hipos}, and
$(\W,\sigma,\R)$ and $(\W\times\WW_2,\zeta\,,\R^+)$ are
as above. Our first result establishes the existence of a global attractor,
which is in addition connected.
\begin{teor}\label{3.teoratt}
Suppose that Hypotheses~\ref{3.hipos} hold.
Then $(\W\times\WW_2,\zeta\,,\R^+)$ is a bounded dissipative semiflow, and it admits a
connected global attractor $\mA$.
\end{teor}
\begin{proof}
Let us rewrite the family \eqref{3.model} as
$\dot y_i(t)=F_i(\wt,y(t),y(t-\tau(y_t))$ for $i=1,2$, and look for a constant
$m$ such that, for all $(\w,z_1,z_2)\in\W\times\R^2\!\times\R^2$ and $i=1,2$,
\begin{equation}\label{3.cot}
 -a_i(\w)\,k_i-m\le F_i(\w,z_1,z_2)\le -a_i(\w)\,k_i+m\,.
\end{equation}
We consider two auxiliary families of (uncoupled) systems of linear ODEs,
\begin{equation}\label{3.maymin}
\left\{
\begin{array}{l}
\dot{y}_1(t)=-a_1(\wt)\,y_1(t)+m\,,\\[.1cm]
\dot{y}_2(t)=-a_2(\wt)\,y_2(t)+m\end{array}\right.
\quad\text{and}\quad
\left\{
\begin{array}{l}
\dot{y}_1(t)=-a_1(\wt)\,y_1(t)-m\,,\\[.1cm]
\dot{y}_2(t)=-a_2(\wt)\,y_2(t)-m\end{array}\right.
\end{equation}
for $\w\in\W$, and note that they satisfy property \hyperlink{Q}{Q}.
Let us represent by
$y^+(t,\w,\alpha)=\lsm y_1^+(t,\w,\alpha)\\ y_2^+(t,\w,\alpha)\rsm$ and
$y^-(t,\w,\alpha)=\lsm y_1^-(t,\w,\alpha)\\ y_2^-(t,\w,\alpha)\rsm$ the
solutions of the left and right systems of \eqref{3.maymin}$_\w$
with $y^+(0,\w,\alpha)=y^-(0,\w,\alpha)=\alpha\in\R^2$.
It follows from \eqref{3.cot} that the map $F=\lsm F_1\\F_2\rsm\colon\W\times\R^2\times\R^2\to\R^2$ giving
rise to \eqref{3.model} can be compared with those of the two families
in \eqref{3.maymin}
in the terms of Theorem~\ref{2.teormonotone}, which consequently guarantees that
\begin{equation}\label{3.compa}
 y^-(t,\w,x(0))\le y(t,\w,x)\le y^+(t,\w,x(0))
\end{equation}
for $(\w,x)\in\W\times\WW_2$ and $t\in[0,\beta_{\w,x})$.
These inequalities and the global existence of
$y^\pm(t,\w,x(0))$ combined with Theorem \ref{2.flujocontinuo}(iii)
guarantee that $\beta_{\w,x}=\infty$ for all $(\w,x)\in\mC^0_2$:
the semiflow $(\W\times\WW_2,\zeta\,,\R^+)$
given by \eqref{3.skew} is globally defined.
\par
Using the fact that $a_i(\w)\ge\delta>0$ for all $\w\in\W$, we can get a
$k_*>0$ such that
\begin{equation}\label{3.cot-a}
 -a_i(\w)\,k+m<-1 \quad \text{for all $\,\w\in\W\,$ and $\,k\ge k_*$}\,,
\end{equation}
so that $-a_i(\w)\,k-m>1$ for all $\w\in\W$ and $k\le-k_*$.
The set $\mK:=\big\{(\w,z_1,z_2)\in\W\times\R^2\times\R^2\,|\;
\lsm -k_*\\-k_*\rsm\le z_i\le\lsm k_*\\k_*\rsm
\;\text{for}\;i=1,2\big\}$ is compact, so that
there exists $d_*$ such that $-d_*\le F_i(\w,z_1,z_2)\le d_*$
for all $(\w,z_1,z_2)\in\mK$ and $i=1,2$.
We define
\[
\begin{split}
 \mS:=\big\{(\w,x)\in\W\times\WW_2\,\big|&\;\lsm -k_*\\-k_*\rsm\le x(s)
 \le \lsm k_*\\k_*\rsm\quad\text{for all $s\in[-r,0]$}\;\text{ and }\\
 &\;\lsm -d_*\\-d_*\rsm \le \dot{x}(s)\le\lsm d_*\\d_*\rsm
 \quad\text{for a.a.~$s\in[-r,0]$}\big\}\,.
\end{split}
\]
The definition of $\mS$ and the monotonicity of the Euclidean norm
on $\R^2$ ensure that $\mS$ is bounded and closed. We will check that it is
$\zeta\,$-absorbing. It is easy to prove
that for any $\alpha\in\R$,
$-k_*\le y_i^-(t,\w,\lsm-\alpha\\-\alpha\rsm)$ and
$y_i^+(t,\w,\lsm\alpha\\\alpha\rsm)\le k_*$
for all $t\ge t_{\alpha}:=\max(0,\alpha-k_*)$,
$\w\in\W$ and $i=1,2\,$:
just note that the inequalities persist to the right of any time $t$ at which
they are satisfied; and if $\alpha>k^*$ use the bounds for the derivatives
provided by \eqref{3.cot-a}.
\par
Let $\mB\subset\W\times\WW_2$ be an arbitrary bounded set, and let us
choose $\alpha_*=\alpha_*(\mB)\in\R$
such that $\lsm-\alpha_*\\-\alpha_*\rsm\le x(0)\le\lsm
\alpha_*\\\alpha_*\rsm $ for all $(\w,x)\in \mB$.
It is a well-known result that the
uncoupled systems \eqref{3.maymin} define monotone global
flows on $\W\times\R^2$ (see Definition~\ref{2.defmonotone}
and Chapter 3 of \cite{smit}).
These facts combined with \eqref{3.compa} yield
\begin{equation}\label{3.compa2}
y^-(t,\w,\lsm-\alpha_*\\-\alpha_*\rsm)\le y^-(t,\w,x(0))\le
 y(t,\w,x)\le y^+(t,\w,x(0))\le y^+(t,\w,\lsm\alpha_*\\\alpha_*\rsm)
\end{equation}
for all $(\w,x)\in\mB$ and $t\ge 0$.
We define $t_{\mB}:=t_{\alpha_*}\ge 0$ and deduce
from the previous paragraph and \eqref{3.compa2} that
$-k_*\le y_i(t,\w,x)\le k_*$ for all $(\w,x)\in\mB$
whenever $t\ge t_\mB$ and for $i=1,2$. This property and the definition of
$d_*$ ensure that $-d_*\le \dot{y}_i(t,\w,x)\le d_*$
for all $(\w,x)\in\mB$ whenever $t\ge t_{\mB}+r$ and $i=1,2$.
That is, $(\wt,u(t,\w,x))\in\mS$ for all
$(\w,x)\in\mB$ and $t\ge t_{\mB}+2r$;
i.e., $\mS$ is $\zeta\,$-absorbing.
In particular, the semiflow $(\W\times\WW_2,\zeta\,,\R^+)$ is bounded dissipative.
\par
Reviewing the argument, we observe that $t_{k_*}=0$,
so that $t_{\mS}=0$. It is easy to deduce that
$\mS$ is positively $\zeta\,$-invariant.
Therefore, $\bigcup_{t\ge 0}\zeta_{\,t}(\mS)=\mS$, so that this union
is bounded. Clearly $\zeta_{\,2r}(\mS)$ is absorbing under $\zeta\,$
and positively $\zeta\,$-invariant (as $\mS$).
Therefore, the closed set
$\mM:={\rm closure\,}_{\W\times\WW_2}\zeta_{\,2r}(\mS)$
is absorbing under $\zeta\,$ and positively
$\zeta\,$-invariant. (Here we make use of Theorem~\ref{2.flujocontinuo}(vii).)
These properties allow us to repeat the argument of the proof of
Theorem~\ref{2.flujocontinuo}(iii) (see Theorem 3.3(ix)
of~\cite{mano1}) in order to check that $\mM$ is compact.
\par
We will now recall how to obtain the global attractor from $\mM$, which
is a standard procedure (despite the lack of global continuity of $\zeta\,$).
Let us define $\mA=\bigcap_{\,t\ge 0}\zeta_{\,t}(\mM)$, which is clearly a
nonempty compact set, and it is positively $\zeta\,$-invariant (since $\mM$ is).
It is easy to check that $(\w,x)$ belongs to $\mA$ if and only if there exist
sequences $((\w_n,x_n))$ in $\mM$ and $(t_n)\uparrow\infty$ such that
$(\w,x)=\lim_{n\to\infty}\zeta(t_n,\w_n,x_n)$, and to deduce that $\mA\subseteq\zeta_{\,t}(\mA)$
for all $t\ge 0$. That is, $\mA$ is $\zeta\,$-invariant. We will now check that
$\mA$ attracts any bounded set $\mB$. Assume for contradiction that there exists
$\ep>0$, $(t_n)\uparrow\infty$, and a sequence $((\w_n,x_n))$ in $\mB$ such that
$\ep\le \dist(\zeta(t_n,\w_n,x_n),\mA)$, where $\dist$ is defined by \eqref{2.dist}
(for a singleton $\mC_1$, and $\mC_2=\mA$).
Since $\zeta(t_n,\w_n,x_n)$ belongs to the (absorbing) set $\mM$
for large enough $n$, there exist subsequences
$(t_m)$ and $((\w_m,x_m))$ and a point $(\w,x)$ with
$(\w,x)=\lim_{m\to\infty}\zeta(t_m,\w_m,x_m)$.
But then $(\w,x)\in\mA$ and $\dist((\w,x),\mA)\ge\ep$, which is impossible.
All these properties show that $\mA$ is a global attractor.
\par
It remains to prove that $\mA$ is connected. Note that it contains a minimal set, and
hence it projects over the whole (minimal) base $\W$, which is connected.
Let us call $\mA^2:=\{x\in\WW_2\,|\;\text{there exists $\w\in\W$ with
$(\w,x)\in\mA$}\}$, which is a compact set, and let
$\widehat \mA^2$ the closed convex hull of $\mA^2$. Then the set
$\W\times \widehat \mA^2$ is compact and connected, and
$\mA\subseteq\widehat\mA^2$.
From this point
we can follow the ideas of Lemma 2.4.1 of \cite{hale4} in order to
prove the connected character of $\mA$. This completes the proof.
\end{proof}
\subsection{Global attractor and pullback attractors}
It is well-know that the definition of the so-called {\em process\/}
associated to a nonautonomous differential equation
provides an approach to analyze its long term dynamics which is different from that of
constructing the hull and the corresponding skew-product semiflow.
Caraballo {\em et al} \cite{calo} combine both dynamical formalisms in the
case of a nonautonomous ODE in order to describe the
properties of the {\em pullback attractor\/}
of the process from the structure of the global attractor
for the skew-product semiflow. In what follows, we make a
similar analysis for the case of our nonautonomous SDDEs:
from those properties of the semiflow $(\W\times\WW_2,\zeta\,,\R^+)$ which
imply the existence of the global attractor $\mA$,
we deduce the existence of the pullback attractor for the initial process
(and for all the processes associated to each one of the equations of
the family), and determine its shape in terms of that of $\mA$.
\par
In what follows we assume that Hypotheses~\ref{3.hipos} hold.
Let us fix, for the moment being, $\w\in\W$, and define
\[
 S_\w(t,s)\colon\WW_2\to\WW_2\,,\;x\mapsto S_\w(t,s)\,x:=u(t-s,\ws,x)
\]
for $t\ge s$. The cocycle equality
$u(t+s,\w,x)=u(t,\ws,u(s,\w,x))$ (which follows from the property
\hyperlink{(2)}{(f2)} for $\zeta\,$),
the continuity of the base flow $(\W,\sigma,\R)$,
and Theorem~\ref{2.flujocontinuo}(vi) and (vii) ensure that:
$S_\w(t,t)=\text{Id}_{\WW_2}$ for all $t\ge 0$;
$S_\w(t,s)=S_\w(t,l)\circ S_\w(l,s)$ whenever $t\ge l\ge s$; if
\[
 \mF_\w:=\{(t,s,x)\in\R\times\R\times\WW_2\,|\;t-s\ge r\}\,,
\]
then $S_\w\colon\mF_\w\to\WW_2,\;(t,s,x)\mapsto S_\w(t,s)\,x$ is continuous
(that is, the map $S_\w$ defines a {\em process}, which can be not continuous
due to the possible lack of continuity for $0\le t-s< r$); and
\begin{equation}\label{3.con}
S_\w(t,s)\colon\WW_2\to\WW_2 \quad\text{is continuous if $s$ and $t$ are fixed}\,.
\end{equation}
A continuous processes associated to the continuous
semiflow $(\mC^0_2,\zeta^0,\R^+)$ given by \eqref{2.defPi0} instead of
$(\W\times\WW_2,\zeta\,,\R^+)$ is defined in Section 3 of
\cite{walth}. But in that case the domain and codomain of
$S_\w(t,s)$, given by sections of the set $\mC^0_2$,
vary with $\w,t$ and $s$. Note that, for our process,
the lack of continuity while $t-s\le r$ is not relevant, since
the long-term analysis is done for a fixed $t$ and $s\to-\infty$.
\par
The distance between two subsets of $\WW_n$ is defined by the analogue of \eqref{2.dist}.
\begin{defi}
A time-dependent family $\{\mK_t\,|\;t\in\R\}$ of compact subsets of $\WW_2$
{\em pullback attracts bounded sets under $S_\w$\/} if $\lim_{s\to-\infty}
\dist(S_\w(t,s)\,\mB,\mK_t)=0$ for all $t\in\R$ for every bounded set $\mB$;
and it is {\em $S_\w$-invariant\/} if $S_\w(t,s)\,\mK_s=\mK_t$ whenever $t\ge s$.
An $S_\w$-invariant family $\{\mK_t\,|\;t\in\R\}$ of compact sets is a
{\em pullback attractor of $S_\w$} if it pullback attracts bounded sets under $S_\w$
and, in addition, it is the minimal $S_\w$-invariant family of
compact sets with this property (in the sense of set inclusion).
And given a bounded set $\mB\subset\WW_2$, its {\em pullback omega-limit set
in time\/} $t\in\R$ is
\[
 \mO_\w(\mB,t):=\{x\in\WW_2\,|\;x=\lim_{n\to\infty}S_\w(t,s_n)\,x_n
 \quad \text{for $(s_n)\downarrow-\infty$ and $(x_n)$ in $\mB$}\}\,.
\]
\end{defi}
Note that the minimality required in the definition of the pullback attractor
implies its uniqueness, in the case of existence.
\par
\begin{prop}\label{3.pull}
Suppose that Hypotheses~\ref{3.hipos} hold.
Let $\mB\subset\WW_2$ be a bounded set. Then, for each $\w\in\W$ and all
$t\in\R$, the set $\mO_\w(\mB,t)$ is nonempty and compact,
and $\lim_{s\to-\infty}\dist(S_\w(t,s)\,\mB,\mO_\w(\mB,t))=0$. In addition,
$S_\w(t,s)\,\mO_\w(\mB,s)=\mO_\w(\mB,t)$ for each $\w\in\W$ whenever $t\ge s$.
\end{prop}
\begin{proof}
Let us fix any $n\in\N$ and take the sets $\mS$ and $\mM$
defined in the proof of Theorem~\ref{3.teoratt},
so that $\mM=\text{closure}_{\W\times\WW_2}\zeta_{\,2r}(\mS)$.
Let us define $\mM_\W:=\{x\in\WW_2\,|\;
\text{there exists }\w\in\W \text{ with }(\w,x)\in\mM\}$,
which is a compact subset of $\WW_2$.
By repeating the ideas of the proof of Theorem~\ref{3.teoratt}
we show that, if $\mB\subset\WW_2$ is bounded,
then there exists $t_\mB$ such that
$S_\w(t,s)\,\mB\subseteq\mS$ for $t-s\ge t_\mB$ and hence that
$S_\w(t,s)\,\mB\subseteq\mM_\W$ for $t-s\ge t_\mB+2r$, for all $\w\in\W$.
In particular, $O_\w(\mB,t)\subseteq\mM_\W$ for all $\w\in\W$ and all $t\in\R$.
In these conditions, we can repeat the proof of
Lemma 2.7 of \cite{chlr}, whose conclusions are those of
this proposition. Note that Lemma 2.7 of \cite{chlr}
relies on Lemma 2.4 of the same book, and that also this result
can be adapted to our setting, due to \eqref{3.con}.
\end{proof}
Recall that $\mA_\w:=\{x\in\WW_2\,|\;(\w,x)\in\mA\}$ for any $\w\in\W$.
\begin{teor}\label{3.teorpull}
Suppose that Hypotheses~\ref{3.hipos} hold, and
let $\mA$ be the global attractor provided by Theorem~\ref{3.teoratt}.
Then, for each $\w\in\W$ and all $t\in\R$,
\[
 \mA_{\wt}=\text{\rm closure}_{\WW_2}
  \left(\bigcup_{\mB\subset\WW_2,\,\mB\;\text{\rm bounded}}
 \mO_\w(\mB,t)\right).
\]
In addition, $\{\mA_{\wt}\,|\;t\in\R\}$ is an $S_\w$-invariant family of compact sets
and it is the pullback attractor of the process $S_\w$.
\end{teor}
\begin{proof}
We fix $\w\in\W$ and define
$ \mK_t:=\text{closure}_{\WW_2}
 \left(\bigcup_{\mB\subset\WW_2,\,\mB\;\text{\rm bounded}}
 \mO_\w(\mB,t)\right)
$.
Let $\mM_\W$ be defined as in the proof of Proposition \ref{3.pull}.
As said there, $\mO_\w(\mB,t)\subseteq\mM_\W$ for all $t\in\R$,
so that $\mK_t\subseteq\mM_\W$ and hence it is a compact set.
It follows from Proposition~\ref{3.pull}
that the family $\{\mK_t\,|\;t\in\R\}$
of compact sets pullback-attracts bounded sets.
In addition, $S_\w(t,s)\,\mK_s=\mK_t$ for $t\ge s$: this property follows
from Proposition~\ref{3.pull} and \eqref{3.con}.
Therefore, the family $\{\mK_t\,|\;t\in\R\}$ is $S_\w$-invariant.
\par
In addition, if $\{\mP_t\,|\;t\in\R\}$ is another $S_\w$-invariant family
of compact sets which attracts bounded sets, then
$\mO_\w(\mB,t)\subseteq\mP_t$ for every bounded set $\mB\subset\WW_2$.
This property follows easily from the definition of $\mO_\w(\mB,t)$
and a contradiction argument. Therefore, $\mK_t\subseteq\mP_t$. Consequently,
$\{\mK_t\,|\;t\in\R\}$ is the pullback attractor.
\par
Let us finally check that $\mK_t=\mA_{\wt}$ for all $t\in\R$.
According to Remarks~\ref{2.notas}.3 and \ref{2.notapseudo}.1,
a point $x\in\WW_2$ belongs to $\mA_\w$
if and only if $(\w,x)$ admits a bounded complete orbit.
Since $\mK_t\subseteq\mM_\W$ for all $t\in\R$,
Theorem 1.17 of \cite{chlr} (whose proof can be repeated without changes for
our process) proves that this is the same
condition required on the point $x\in\WW_2$ in order to belong to $\mK_0$.
Therefore, $\mK_0=\mA_\w$. The same argument works for every $t\in\R$,
and this completes the proof.
\end{proof}
\subsection{Properties of the global attractor}
We continue working under Hypotheses~\ref{3.hipos}, so that
Theorem~\ref{3.teoratt} ensures the existence of the global attractor $\mA$
for the semiflow $(\W\times\WW_2,\R^+,\zeta)$.
The restriction of $\zeta\,$ to the global attractor $\mA$,
which is continuous (see Theorem \ref{2.flujocontinuo}(viii)),
determines the long-term
behaviour of the bounded semiorbits of $(\W\times\WW_2,\zeta\,,\R^+)$.
The following goal is to describe conditions which provide $\mA$ with the simplest
possible structure: that of a copy of the base,
which means that $\mA$ agrees with the graph
of a continuous map $a\colon\W\to\WW_2$. This goal
is achieved in Theorem \ref{3.teorAcopy},
whose hypotheses consist of Hypotheses~\ref{3.hipos} together
with the exponential stability of all the $\zeta\,$-minimal sets.
\begin{nota}\label{2.notaLyap}
We recall two facts: the exponential stability of a $\zeta\,$-minimal set $\mM$
is characterized in Section 5 of~\cite{mano2}
by the negative character of its upper Lyapunov exponent
with respect to $\zeta\,$ (see Definition~\ref{3.defly} and Theorem~\ref{3.teor52} below);
and all the $\zeta\,$-minimal sets are contained in the global attractor $\mA$
(see Remarks~\ref{2.notapseudo}.1 and \ref{2.notas}.2). That is,
the hypothesis of exponential stability of all the $\zeta\,$-minimal sets
holds if the upper Lyapunov exponent of $\mA$ is negative, which
a priori seems to be a more restrictive property. But
Theorem \ref{3.teorAcopy}(iii) shows that
both conditions are equivalent.
\end{nota}
\begin{teor}\label{3.teorAcopy}
Suppose that Hypotheses~\ref{3.hipos} hold, and that all the $\zeta\,$-minimal
subsets of $\W\times\WW_2$ are exponentially stable.
Let $\mA$ be the global attractor provided by Theorem~\ref{3.teoratt}. Then,
\begin{itemize}
\item[(i)] $\mA$ is $\zeta\,$-minimal and the continuous
semiflow $(\mA,\zeta\,,\R^+)$ admits a flow extension.
In particular, $\mA$ is the unique $\zeta\,$-minimal subset of $\W\times\WW_2$.
\item[(ii)] $\mA$ is a copy of the base; i.e., there exists a continuous function
  $a\colon\W\to\WW_2$ such that $\mA=\{(\w,a(\w))\,|\;\w\in\W\}$.
\item[(iii)] There exist a constant
  $\beta>0$ and, for any $c>0$, a constant $k_c\ge 1$ such that, if
  $x\in\WW_2$ satisfies $\n{x}_{\WW_2}\le c$, then
  \[
  \qquad\n{u(t,\w,x)-a(\wt)}_{\WW_2}\le k_c\,e^{-\beta t}\n{x-a(\w)}_{\WW_2}
  \quad\text{for all $t\ge 0$}\,.
  \]
\end{itemize}
\end{teor}
\begin{proof}
(i) Recall that Hypotheses~\ref{3.hipos} include the minimality of $(\W,\sigma,\R)$.
According to Theorem 5.9 of \cite{mano2},
the exponential stability of all
the $\zeta\,$-minimal sets (i.e., the strict negativeness of
the upper Lyapunov exponent of all the $\zeta\,$-minimal sets), together
with the fact that $\mA$ is connected
(see Theorem~\ref{3.teoratt}), ensures that $\mA$
contains at most a $\zeta\,$-minimal set $\mM$
which agrees with the omega-limit set of any of the elements of $\mA$.
In particular, the semiflow $(\W\times\WW_2,\zeta\,,\R^+)$ admits just one $\zeta\,$-minimal set.
In addition, according to Corollary 5.7 of~\cite{mano2}, $\mM$ is an $m$-cover of the base
for an integer $m\ge 1$ admitting a flow extension.
Therefore, (i) will be proved once shown that $\mA\subseteq\mM$.
\par
Recall that the exponential
stability of $\mM$ means that there exist $\beta>0$, $k\ge 1$,
and $\delta>0$ such that, if $(\w,\bar x)\in\mM$ and $(\w,x)\in\W\times\WW_2$
satisfy $\n{x-\bar x}_{\WW_2}<\delta$, then the function $u(t,\w,x)$ is defined for
$t\in[0,\infty)$, and
\begin{equation}\label{3.exp}
 \n{u(t,\w,x)-u(t,\w,\bar x)}_{\WW_2}\le k\,e^{-\beta t}\,
 \n{x-\bar x}_{\WW_2}\quad\text{for all $t\ge 0$}\,.
\end{equation}
\par
On the other hand, Corollary 5.7 of~\cite{mano2} (based on previous results
of \cite{SackerSell} and \cite{noos5}) states that
for each $\w_0\in\W$ there exist a neighborhood $\mU_{\,\w_0}\subseteq\W$ of
$\w_0$ and $m$ continuous maps $a^1_{\w_0},\ldots,a^m_{\w_0}
\colon\mU_{\,\w_0}\to\WW_2$ such that
\begin{equation}\label{3.fibra}
 \mM_{\w}:=\{x\in\WW_2\,|\;(\w,x)\in\mM\}=\{a^1_{\w_0}(\w),\ldots,a^m_{\w_0}(\w)\}
\end{equation}
for all $\w\in\mU_{\,\w_0}$, with $a^i_{\w_0}(\w)\ne a^j_{\w_0}(\w)$ whenever
$\w\in\mU_{\,\w_0}$ and $i\ne j$.
\par
We take $(\wit\w,\wit x)\in\mA$ and a negative semiorbit of it
in $\mA$ (see Remark \ref{2.notas}.3), which we represent by
$\{(\wit\w\pu t,\wit x_t)\,|\,t\le 0\}$. The alpha-limit of $(\wit\w,\wit x)$
is a positively $\zeta\,$-invariant compact set,
and hence it contains the unique $\zeta\,$-minimal
set $\mM$. Now we take $(\w_0,x_0)\in\mM$, assume without restriction that
$x_0=a^1_{\w_0}(\w_0)$,
and take a $\rho>0$ such that if $d_\W(\w_0,\w)\le\rho$ then $\w\in\mU_{\,\w_0}$ and
$\n{a^1_{\w_0}(\w_0)-a^1_{\w_0}(\w)}_{\WW_2}\le\delta/2$.
Let us write $(\w_0,x_0)=\lim_{n\to\infty}
(\wit\w\pu(-t_n),\wit x_{-t_n})$ for a suitable sequence $(t_n)\uparrow\infty$, and
take $n_0$ such that, for all $n\ge n_0$, the following two properties hold:
$d_\W(\w_0,\wit\w\pu(-t_n))\le\rho$ and $\n{x_0-\wit x_{-t_n}}_{\WW_2}<\delta/2$.
It follows immediately that  $\n{\wit x_{-t_n}-a^1_{\w_0}(\wit\w\pu(-t_n))}_{\WW_2}\le
\n{\wit x_{-t_n}-a^1_{\w_0}(\w_0)}_{\WW_2}+
\n{a^1_{\w_0}(\w_0)-a^1_{\w_0}(\wit\w\pu(-t_n))}_{\WW_2}\le \delta$ for $n\ge n_0$.
Consequently, by~\eqref{3.exp},
\[
\begin{split}
 &\n{\wit x-u(t_n,\wit\w\pu(-t_n),a^1_{\w_0}(\wit\w\pu(-t_n)))}_{\WW_2}\\
 &\quad=\n{u(t_n,\wit\w\pu(-t_n),\wit x_{-t_n})-
 u(t_n,\wit\w\pu(-t_n),a^1_{\w_0}(\wit\w\pu(-t_n)))}_{\WW_2}
 \le k\,e^{-\beta t_n}\,\delta\,,
\end{split}
\]
which implies that
\[
 (\wit\w,\wit x)
 =\!\lim_{n\to\infty}(\wit\w,u(t_n,\wit\w\pu(-t_n),a^1_{\w_0}(\wit\w\pu(-t_n))))
 =\!\lim_{n\to\infty}\zeta(t_n,\wit\w\pu(-t_n),a^1_{\w_0}(\wit\w\pu(-t_n)))\,.
\]
The points $(\wit\w\pu(-t_n),a^1_{\w_0}(\wit\w\pu(-t_n))$
belong to $\mM$, and consequently
also the points $\zeta(t_n,\wit\w\pu(-t_n),a^1_{\w_0}(\wit\w\pu(-t_n))$
belong to $\mM$. Hence, $(\wit\w,\wit x)\in\mM$, and (i) is proved.
\smallskip\par
(ii)
As said in the proof of (i), $\mA=\mM$ is an exponentially stable $m$-cover of the base.
We will use the notation established in~\eqref{3.fibra}, having in mind
that $\mA_\w=\mM_\w$, and choosing (without restriction)
the neighborhood $\mU_{\,\w_0}$ of each $\w_0\in\W$ to be compact.
Let us assume for contradiction that $m\ge 2$. Then, there exists
$d_*>0$ such that, for any $\w\in\W$, any two points of $\mA_\w$ are at a distance
greater than $d_*$. This assertion follows from the compactness of $\mU_{\,\w_0}$,
which in turn ensures that
$d_{\w_0}\!:=\min_{1\le i<j\le m,\w\in\mU_{\,\w_0}}
\n{a^i_{\w_0}(\w)-a^j_{\w_0}(\w)}_{\WW_2}>0$ for all $\w_0\in\W$,
and from the compactness of $\W$. 
The constant $d_*$
will play a key role in what follows.
\par
Theorem 3.3 of \cite{noos5} proves that the map
$\w\mapsto\mA_\w$ is continuous in the Hausdorff topology of the set of compact subsets
of $\WW_2$. This fact ensures the existence of $\rho_1>0$ such that
\begin{equation}\label{3.dd}
\text{if $\;d_\W(\w_1,\w_2)\le\rho_1$}\quad\text{ then
$\;\max_{x_1\in\mA_{\w_1}}\left(\min_{x_2\in\mA_{\w_2}}
 \n{x_1-x_2}_{\WW_2}\right)\le d_*/6\;$:}
\end{equation}
see definition~\eqref{2.distreal}. In other words, if $d_\W(\w_1,\w_2)\le\rho_1$
then for every $x_1\in\mA_{\w_1}$ there exists at least a point $x_2\in\mA_{\w_2}$
such that $\n{x_1-x_2}_{\WW_2}\le d_*/6$.
We can assume without restriction that $\rho_1\le d_*/3$.
\par
Let us fix $\w_0\in\W$, choose $x_0$ and $x_1$ in $\mA_{\w_0}$ with $x_0\ne x_1$,
and define $x_\alpha=\alpha\,x_1+(1-\alpha)\,x_0$ for all $\alpha\in[0, 1]$.
Choose also a sequence $(t_n)\uparrow\infty$ such that $\lim_{n\to\infty}\w_0\pu t_n=\w_0$,
and choose $\rho_2\in(0,\rho_1]$ such that, if $d_\W(\w,\w_0)\le\rho_2$, then
$\w\in\mU_{\,\w_0}$. Since $\mA$ is the global attractor and the set
$\mB:=\{(\w_0,x_\alpha)\,|\,\alpha\in[0, 1]\}$ is bounded,
there exists $n_0$ large enough to guarantee that
$d_\W(\w_0\pu t^*\!,\w_0)\le\rho_2/2$ and
$\dist(\zeta_{\,t^*}(\mB),\mA)<\rho_2/2$ for $t^*=t_{n_0}>0$.
The definition of $\dist$ (see \eqref{2.dist}) ensures that,
for each $\alpha\in[0,1]$ there exists
$(\w_\alpha,x^*_\alpha)\in\mA$ with
$d_\W(\w_0\pu t^*\!,\w_\alpha)+\n{u(t^*\!,\w_0,x_\alpha)-x^*_\alpha}_{\WW_2}\le\rho_2/2$.
In particular, for each $\alpha\in[0,1]$, $d_\W(\w_0,\w_\alpha)\le\rho_2$, and hence
there exists $j(\alpha)\in\{1,\ldots,m\}$
such that $x^*_\alpha=a^{j(\alpha)}_{\w_0}(\w_\alpha)$; and, since
$d_\W(\w_0\pu t^*\!,\w_\alpha)\le\rho_1$, we deduce from \eqref{3.dd}
that there exists $i(\alpha)\in\{1,\ldots,m\}$ with
$\n{a^{i(\alpha)}_{\w_0}(\w_0\pu t^*)-a^{j(\alpha)}_{\w_0}(\w_\alpha)}_{\WW_2}\le d_*/6$.
Therefore,
\begin{equation}\label{3.fundbound}
\begin{split}
 &\n{u(t^*\!,\w_0,x_\alpha)-a^{i(\alpha)}_{\w_0}(\w_0\pu t^*)}_{\WW_2}
 \le\n{u(t^*\!,\w_0,x_\alpha)-a^{j(\alpha)}_{\w_0}(\w_\alpha)}_{\WW_2}\\
 &\qquad\qquad\qquad\qquad\quad +
 \n{a^{j(\alpha)}_{\w_0}(\w_\alpha)-a^{i(\alpha)}_{\w_0}(\w_0\pu t^*)}_{\WW_2}
 \le \Frac{\rho_2}{2} + \Frac{d_*}{6}\le \Frac{d_*}3
\end{split}
\end{equation}
for all $\alpha\in[0,1]$.
Note also that $i(\alpha)$ is the unique index satisfying
the previous bound: the existence
of two of them would provide two points on $\mA_{\w_0\pu t^*}$
at a distance no greater than
$2d_*/3$, which is impossible. In particular,
\begin{equation}\label{3.esta}
 \!\!(\w_0\pu t^*\!,a_{\w_0}^{i(0)}(\w_0\pu t^*))=\zeta(t^*\!,\w_0,x_0)
 \;\text{\,and\,}\;
 (\w_0\pu t^*\!,a_{\w_0}^{i(1)}(\w_0\pu t^*))=\zeta(t^*\!,\w_0,x_1)\,,
\end{equation}
since all these points belong to the $\zeta\,$-invariant set $\mA$.
\par
The next goal is to check that the map $\alpha\mapsto i(\alpha)$ is constant on
$[0,1]$, which is the same as saying that it is locally constant. Note that
there is no restriction in assuming that $t^*\ge r$.
We fix $\alpha_0\in[0,1]$
and use the continuity of the semiflow $\zeta\,$ guaranteed by
Theorem \ref{2.flujocontinuo}(vi) to find $\ep>0$ such that, if
$\alpha\in(\alpha_0-\ep,\alpha_0+\ep)\cap[0,1]$, then
$\n{u(t^*,\w,x_{\alpha_0})-u(t^*,\w,x_\alpha)}_{\WW_2}<d_*/3$.
This inequality, together with \eqref{3.fundbound}, means that
$\n{a^{i(\alpha_0)}_{\w_0}(\w_0\pu t^*)-a^{i(\alpha)}_{\w_0}(\w_0\pu t^*)}_{\WW_2}<d_*$,
and the assertion follows once more from the definition of $d_*$.
\par
The last property and \eqref{3.esta} prove that $\zeta(t^*\!,\w_0,x_0)
=\zeta(t^*\!,\w_0,x_1)$,
which is impossible, since $\zeta\,$ defines a flow on $\mA$ and $x_0\ne x_1$. This
contradiction proves that $m=1$. The existence of the continuous map
$a\colon\W\to\WW_2$ such that $\mA=\{(\w,a(\w))\,|\;\w\in\W\}$
is a trivial consequence of this fact (recall that $\mA$ is closed),
and the proof of (ii) is complete.
\smallskip\par
(iii) Note that $u(t,\w,a(\w))=a(\wt)$ for all $t\in\R$ and $\w\in\W$.
For any $c> 0$,
we define $\mB_c:=\{x\in\WW_2\,|\;\n{x}_{\WW_2}\le c\}$,
which is a bounded subset of $\WW_2$.
Since $\mA$ is the global attractor, there exists $t_c$ such that
$\dist(\zeta_{\,t_c}(\W\times\mB_c),\mA)<\delta$, where $\delta$ is determined by~\eqref{3.exp}.
Therefore, if $x\in\mB_c$,
\[
\begin{split}
 \n{u(t,\w,x)-a(\wt)}_{\WW_2}&\!=\!
 \n{u(t-t_c,\wt_c,u(t_c,\w,x))\!-\!u(t-t_c,\wt_c,a(\wt_c))}_{\WW_2}\\
 &\le \!k\,e^{-\beta(t-t_c)}\,
 \n{u(t_c,\w,x)-a(\wt_c)}_{\WW_2}
\end{split}
\]
for all $t\ge t_c$. The properties guaranteed by Hypotheses~\ref{3.hipos}
ensure the existence of $l_c\ge 1$ such that $\n{u(t,\w,x)-a(\wt)}_{\WW_2}\le l_c
\n{x-a(\w)}_{\WW_2}$ whenever $(\w,x)\in\W\times\mB_c$ for all $t\in[0,t_c]$:
see Theorem 3.6 of \cite{mano1}.
The assertion in (iii) is satisfied by $k_c:=k\,l_ce^{\beta t_c}\ge l_c\,$.
\end{proof}
\begin{nota}
Note that, since the orbits of $(\W,\sigma,\R)$ are
almost periodic (and with the same frequency module:
see e.g.~\cite{jomo}),
the existence of a copy of the base ensures the existence of at least one
almost periodic solution for each one of the systems of the family, including of
course the initial one \eqref{3.model0}. If, in addition, the copy of the base
is a global attractor, then
this almost periodic solution is unique (this follows, for
example, from Remark~\ref{2.notas}.2); and, if this attractor is
exponentially stable, so is the almost periodic solution. Altogether,
these properties read as:
if the hypotheses of Theorem~\ref{3.teorAcopy} hold, then
the initial system \eqref{3.model0} has a unique almost periodic solution, which
in addition is exponentially stable.
\end{nota}
\subsection{Bounding the upper Lyapunov exponents}
As we advanced in Remark \ref{2.notaLyap}
(and will formulate explicitly in Theorem~\ref{3.teor52}),
the hypotheses regarding exponential stability of Theorem~\ref{3.teorAcopy} hold
if and only if the upper Lyapunov exponents of the
$\zeta\,$-minimal sets are negative (in which case the unique minimal set is the
global attractor $\mA$). So that the obvious
question is how to obtain these exponents, or
at least how to bound them. The rest of this section is focussed on this problem.
\par
Let us first recall the definition of the upper Lyapunov exponent of a
$\zeta\,$-invariant compact subset $\mK\subset\W\times\WW_2$.
The set $\mK$ could be, for instance, any $\zeta\,$-minimal set or omega-limit set,
and either is contained in the global attractor $\mA$ provided
by Theorem~\ref{3.teoratt} under Hypotheses \ref{3.hipos}, or it is the
set $\mA$ itself: see Remarks~\ref{2.notas}.1 and \ref{2.notapseudo}.1.
According to Theorem \ref{2.flujocontinuo}(viii), $(\mK,\zeta\,,\R^+)$ is a continuous
semiflow.
We represent the elements of $\mA$ by $\hw=(\w,x)$ with $x=\lsm x_1\\x_2\rsm$,
and note that Remark \ref{2.notapseudo}.2 ensures that $x\in C^1_2\subset\WW_2$.
We write $\hwt=\zeta(t,\hw)$ for $\hw\in\mA$ and $t\ge 0$,
and consider the family of linear variational equations
\begin{equation}\label{3.eqvariacional}
\left\{\begin{array}{l}
 \dot{z}_1(t)=A_1(\hwt)\,z_1(t)+B_1(\hwt)\,z_2(t)+C_1(\hwt)\,z_2(t-\wih\tau(\hwt))\\[.1cm]
 \dot{z}_2(t)=A_2(\hwt)\,z_2(t)+B_2(\hwt)\,z_1(t)+C_2(\hwt)\,z_1(t-\wih\tau(\hwt))
\end{array}\right.
\end{equation}
for $\hw=(\w,x)\in\mA$, with
$\wih\tau\colon\mK\to[0,r]\,,\;\hw\mapsto\tau(x_1(0),x_2(0))$, and
where $A_1,\,B_1,\,C_1,$ $A_2,$ $\,B_2,\,C_2\colon\mA\to\R$ are defined by
\[
\begin{split}
A_1(\hw)&:=-a_1(\w)+b_1(\w)\dot{f}(x_1(0))\\
&\quad\, -c_1(\w)\,\dot{g}(x_2(-\tau(x_1(0), x_2(0))))\,
\dot{x_2}(-\tau(x_1(0),x_2(0)))\,\Frac{d\tau}{dy_1}(x_1(0),x_2(0))\,,\\
B_1(\hw)&:=-c_1(\w)\,\dot{g}(x_2(-\tau(x_1(0),x_2(0))))\,
\dot{x_2}(-\tau(x_1(0),x_2(0)))\,\Frac{d\tau}{dy_2}(x_1(0), x_2(0))\,,\\
C_1(\hw)&:=c_1(\w)\,\dot{g}(x_2(-\tau(x_1(0),x_2(0))))\,,
\end{split}
\]
\[
\begin{split}
A_2(\hw)&:=-a_2(\w)+\,\,c_2(\w)\,\dot{g}(x_2(0))\\
&\quad\, -b_2(\w)\,\dot{f}(x_1(-\tau(x_1(0), x_2(0))))\,
\dot{x_1}(-\tau(x_1(0)\,x_2(0)))\,\Frac{d\tau}{dy_2}(x_1(0), x_2(0))\,,\\
B_2(\hw)&:=-b_2(\w)\,\dot{f}(x_1(-\tau(x_1(0), x_2(0))))\,
\dot{x_1}(-\tau(x_1(0),x_2(0)))\,\Frac{d\tau}{dy_1}(x_1(0),x_2(0))\,,\\
C_2(\hw)&:=b_2(\w)\,\dot{f}(x_1(-\tau(x_1(0),x_2(0))))\,.
\end{split}
\]
As is explained in Sections 4 of \cite{mano1} and 3 of~\cite{mano2}, the family
\eqref{3.eqvariacional} induces two globally defined linear skew-product semiflows,
namely
\begin{equation}\label{4.twoskew}
\begin{split}
 &\zeta_L\colon\R^+\!\!\times\mA\times\WW_2\to\mA\times\WW_2\,,
 \quad(t,\hw, v)\mapsto(\hwt,w(t,\hw,v))\,,\\
 & \wit\zeta_L\colon\R^+\!\!\times\mA\times C_2\to\mA\times C_2\,,\qquad\quad\;\;\;
 (t,\hw, v)\mapsto(\hwt,w(t,\hw,v))\,.
\end{split}
\end{equation}
In both cases, $w(t,\hw,v)(s):=z(t+s,\hw,v)$ for
$s\in[-r,0]$, where $z(t,\hw,v)$ represents the solution of
the system \eqref{3.eqvariacional} corresponding to $\hw\in\mA$ with
$z(s,\hw,v)=v(s)$ for $s\in[-r,0]$. (And, in fact, $w(t,\w,x,v)$
determines the derivative with respect to the initial condition $x$
in the direction of the vector $v$ of the solution $u(t,\w,x)$ lying in
the set $\mA$; that is,
$w(t,\w,x,v)=u_x(t,\w,x)\,v=\lim_{\ep\to 0}(u(t,\w,x+\ep v)-u(t,\w,x))/\ep$.)
Note that the difference
between $\zeta_L$ and $\wit\zeta_L$ relies on their domains of definition: the restriction
of $\wit\zeta_L$ to $\mA\times\WW_2$ agrees with $\zeta_L$. In fact,
$(\mA\times C_2,\wit\zeta_L,\R^+)$ is a continuous semiflow, as
Corollary 4.3 of \cite{mano1} proves; while
$(\mA\times\WW_2,\zeta_L,\R^+)$ satisfies similar properties
to those described in Theorem \ref{2.flujocontinuo}, which are also detailed in
Corollary 4.3 of \cite{mano1}.
\par
Let $\mK$ be a $\zeta\,$-invariant compact set. Then
$\zeta_L$ and $\wit\zeta_L$ can be restricted to $\mK\times\WW_2$ and
$\mK\times C_2$.
We represent by $\lambda_\mK$ and $\wit\lambda_\mK$ the upper Lyapunov exponents
of $\mK$ for the semiflows $(\mK\times\WW_2\!,\zeta_L,\R^+)$ and
$(\mK\times C_2,\wit\zeta_L,\R^+)$, respectively (see Definition~\ref{2.defLyap},
and note that it can be adapted without changes to the case of
the semiflow $(\mA\times\WW_2,\zeta_L,\R^+)$, as explained in
Remark 3.9 of \cite{mano2}).
The following property, fundamental for our purposes, is proved in
Theorem 3.10 of~\cite{mano2}.
\begin{teor}\label{3.teorigual}
Suppose that Hypotheses~\ref{3.hipos} hold. Then, $\lambda_\mK=\wit\lambda_\mK$.
\end{teor}
\begin{defi}\label{3.defly}
Suppose that Hypotheses~\ref{3.hipos} hold.
The {\em upper Lyapunov exponent of the set $\mK$
for the semiflow $(\mK,\zeta\,,\R^+)$}
is $\lambda_\mK=\wit\lambda_\mK$.
\end{defi}
The next result, previously mentioned, is part of Theorem 5.2 of \cite{mano2}.
\begin{teor}\label{3.teor52}
Suppose that Hypotheses~\ref{3.hipos} hold, and
let $\mM\subset\W\times\WW_2$ be a minimal set. Then, $\mM$ is
exponentially stable if and only if $\lambda_\mM<0$.
\end{teor}
The difficulty to estimate the value of the upper-Lyapunov exponent is obvious
in the case that we are considering:
the equations of the family \eqref{3.eqvariacional} are written in terms of
the semiorbits of $(\mA,\zeta\,,\R^+)$, which in general are unknown.
So that our goal is to bound $\lambda_\mK$ (for all $\mK$ as described before)
by a quantity which is easier to estimate. Our method
is based on results of comparison of solutions applied to the
equations of the family \eqref{3.eqvariacional} and those of two
families of cooperative linear systems of FDEs (see Definition~\ref{2.defmonotone}):
\eqref{3.eqvarcomp} and \eqref{3.eqvarmay}. The advantages of
estimating the upper Lyapunov exponents for these families
are explained in Remark~\ref{3.notaventajas}.
\par
So, we first consider the family of FDEs
\begin{equation}\label{3.eqvarcomp}
\left\{\!\!\begin{array}{l}
 \dot{z}_1(t)=A_1(\hwt)\,z_1(t)+|B_1(\hwt)|\,
 z_2(t)+|C_1(\hwt)|\,z_2(t-\wih\tau(\hwt))\,,\\[0.2cm]
 \dot{z}_2(t)=A_2(\hwt)\,z_2(t)+|B_2(\hwt)|\,z_1(t)+|C_2(\hwt)|\,z_1(t-\wih\tau(\hwt))
\end{array}\right.
\end{equation}
for $\hw\in\mA$, which we write for short as $\dot z(t)=\bar L(\hwt)\,z_t$.
For each $v\in C_2$, we denote by $\bar{z}(t,\hw,v)$ the solution of \eqref{3.eqvarcomp}
satisfying $\bar z(s,\hw,v)=v(s)$ for $s\in[-r,0]$, and
define $\bar w(t,\hw,v)(s)=\bar z(t+s,\hw,v)$ for $s\in[-r,0]$. Then the map
\[
 \overline{\zeta}_L\colon\R^+\!\!\times\mA\times C_2\to \mA\times C_2\,,\quad
 (t,\hw,v)\mapsto(\hwt, \bar w(t,\hw,v))
\]
defines a linear skew-product semiflow on $\mA\times C_2$.
It is easy to check that, for all $\hw\in\mA$, the vector field
$(t,z)\mapsto \bar L(\hwt)\,z_t$ satisfies
the quasimonotonicity condition on $\R\times C_2$ described in
Chapter 5 of~\cite{smit} (which is the same as
our condition \hyperlink{2.QM}{Q} but for $x^1$ and $x^2$ in $C_2$).
Therefore, Theorem 1.1 of the same chapter
shows that the semiflow $(\mA\times C_2,\overline\zeta_L,\R^+)$
is monotone (see Definition~\ref{2.defmonotone}).
Clearly, $\overline\zeta_L$ can be restricted to $\mK\times C_2$.
We represent by $\bar{\lambda}_{\mK}$ the upper Lyapunov exponent of $\mK$ for the
semiflow $(\mK\times C_2,\overline\zeta_L,\R^+)$.
\par
The following notation will be used from now on: given a two-dimensional
real vector or function $x=\lsm x_1\\ x_2\rsm$,
we denote $|x|_m=\lsm|x_1|\\|x_2|\rsm$. Note that $|v|_m\in C_2$ if $v\in C_2$, and that
$\n{v}_{C_2}=\n{|v|_m}_{C_2}$.
\begin{teor}\label{3.teorcomp1}
Suppose that Hypotheses~\ref{3.hipos} hold. Then,
\begin{itemize}
\item[(i)] $|z(t,\hw,v)|_m\le \bar{z}(t, \hw,|v|_m)$ for all $v\in C_2$.
\item[(ii)] If $\mM\subset\W\times\WW_2$ is a $\zeta\,$-minimal set, then
$\lambda_\mM\le\bar\lambda_\mM$.
\item[(iii)] If $\bar\lambda_\mM<0$ for
all the $\zeta\,$-minimal sets $\mM\subset\W\times\WW_2$, then
all the conclusions of Theorem \ref{3.teorAcopy} hold.
\end{itemize}
\end{teor}
\begin{proof}
Statement (i) is proved in Lemma 4.2 of Novo {\em et al.}~\cite{noos4}, and
ensures that
\[
 \limsup_{t\to\infty}\frac{1}{t}\ln\n{z_t(\hw, v)}_{C_2}
 \le \limsup_{t\to\infty}\frac1{t}\ln\n{\bar{z}_t(\hw, |v|_m)}_{C_2}\,.
\]
This inequality, the definition of the upper Lyapunov exponent, and
Theorem~\ref{3.teorigual}, yield (ii). Statement (iii) is a consequence
of (ii) and Theorems \ref{3.teor52} and \ref{3.teorAcopy}.
\end{proof}
Given a function $l\colon\W\to\R$, we denote
$l^+(\w):=\max(l(\w),0)$ and $l^-(\w):=\max(-l(\w),0)$, so
that $l=l^+-l^-$. The second \lq\lq majorant\rq\rq~family
of cooperative FDEs is
\begin{equation}\label{3.eqvarmay}
\left\{\!\!\begin{array}{l}
\dot{z}_1(t)=\wih A_1(\wt)\,z_1(t)+\wih B_1(\wt)\,z_2(t)+
\wih C_1(\wt)\,z_2(t-\wih\tau(\hwt))\,,\\[0.2cm]
\dot{z}_2(t)=\wih A_2(\wt)\,z_2(t)+\wih B_2(\wt)\,z_1(t)+
\wih C_2(\wt)\,z_1(t-\wih\tau(\hwt))
\end{array}\right.
\end{equation}
for $\hw=(\w,\lsm x_1\\x_2\rsm)\in\mA$, with
$A_i(\hw)\le \wih A_i(\w)$, $B_i(\hw)\le|B_i(\hw)|\le \wih B_i(\w)$ and
$C_i(\hw)\le|C_i(\hw)|\le \wih C_i(\w)$ for $i=1,2$.
In order to ensure these conditions, we take
\[
\begin{split}
\wih A_1(\w)&:=-a_1(\w)+\kappa_{1f}b_1^+(\w)-\ep_{1f}b_1^-(\w)+
\kappa_{2g}c_1^-(\w)-\ep_{2g}c_1^+(\w)\,,\\
\wih B_1(\w)&:=\kappa_{3g}|c_1(\w)|\,,\\
\wih C_1(\w)&:=\kappa_{4g}|c_1(\w)|\,,\\
\wih A_2(\w)&:=-a_2(\w)+\kappa_{1g}c_2^+(\w)-\ep_{1g}c_2^-(\w)
+\kappa_{2f}b_2^-(\w)-\ep_{2f}b_2^+(\w)\,,\\
\wih B_2(\w)&:=\kappa_{3f}|b_2(\w)|\,,\\
\wih C_2(\w)&:=\kappa_{4f}|b_2(\w)|\,,
\end{split}
\]
where
\[
\begin{split}
\kappa_{1f}&:=\sup_{\hw\in\mA}\dot{f}(x_1(0))\,,\qquad
\ep_{1f}:=\inf_{\hw\in\mA}\dot{f}(x_1(0))\,,\\
\kappa_{1g}&:=\sup_{\hw\in\mA}\dot{g}(x_2(0))\,,\qquad
\ep_{1g}:=\inf_{\hw\in\mA}\dot{g}(x_2(0))\,,\\
\kappa_{2f}&:=\sup_{\hw\in\mA}\dot{f}(x_1(-\tau(x_1(0), x_2(0))))\,
\dot{x_1}(-\tau(x_1(0), x_2(0)))\,\Frac{d\tau}{dy_2}(x_1(0), x_2(0))\,,\\
\ep_{2f}&:=\inf_{\hw\in\mA}\dot{f}(x_1(-\tau(x_1(0), x_2(0))))\,
\dot{x_1}(-\tau(x_1(0), x_2(0)))\,\Frac{d\tau}{dy_2}(x_1(0), x_2(0))\,,\\
\kappa_{2g}&:=\sup_{\hw\in\mA}\dot{g}(x_2(-\tau(x_1(0), x_2(0))))\,
\dot{x_2}(-\tau(x_1(0), x_2(0)))\,\Frac{d\tau}{dy_1}(x_1(0), x_2(0))\,,\\
\ep_{2g}&:=\inf_{\hw\in\mA}\dot{g}(x_2(-\tau(x_1(0), x_2(0))))\,
\dot{x_2}\,(-\tau(x_1(0), x_2(0)))\Frac{d\tau}{dy_1}(x_1(0), x_2(0))\,,\\
\kappa_{3f}&:=\sup_{\hw\in\mA}\dot{f}(x_1(-\tau(x_1(0), x_2(0))))\,\left|\,
\dot{x_1}(-\tau(x_1(0), x_2(0)))\,\Frac{d\tau}{dy_1}(x_1(0), x_2(0))\,\right|\,,\\
\kappa_{3g}&:=\sup_{\hw\in\mA}\dot{g}(x_2(-\tau(x_1(0), x_2(0))))\,\left|\,
\dot{x_2}(-\tau(x_1(0), x_2(0)))\,\Frac{d\tau}{dy_2}(x_1(0), x_2(0))\,\right|\,,\\
\end{split}
\]
\[
\begin{split}
\kappa_{4f}&:=\sup_{\hw\in\mA}\dot{f}(x_1(-\tau(x_1(0), x_2(0))))\,,\\
\kappa_{4g}&:=\sup_{\hw\in\mA}\dot{g}(x_2(-\tau(x_1(0), x_2(0))))\,.
\end{split}
\]
We represent by $\wih{z}(t,\hw,v)$
the solution of \eqref{3.eqvarmay} with
$\wih{z}(t,\hw,v)(s)=v(s)$ for $s\in[-r,0]$
for each $\hw\in\mA$ and $v\in C_2$, denote $\wih w(t,\hw,v)(s)=\wih z(t+s,\hw,v)$
for $s\in[-r,0]$, and consider the linear skew-product semiflow
\[
 \wih\zeta_L\colon\R^+\!\!\times\mA\times C_2\to\mA\times C_2\,,\quad
 (t,\hw, v)\mapsto(\hwt,\wih w(t,\hw,v))\,,
\]
which is also monotone (see again Theorem 1.1 of Chapter 5 of~\cite{smit}).
And, for all $\zeta\,$-minimal set
$\mM\subset\W\times\WW_2$
we represent by $\wih\lambda_\mM$ the upper Lyapunov exponent for $\mM$ with
respect to the restricted semiflow $(\mM\times C_2,\wih\zeta_L,\R^+)$.
\par
Note that, if we represent \eqref{3.eqvarcomp} and \eqref{3.eqvarmay}
by $\dot{z}(t)=\bar L(\hwt)\,z_{t}$ and
$\dot{z}(t)=\wih L(\hwt)\,z_{t}$,
respectively, then
\begin{equation}\label{3.desigL}
 \bar L(\hw)\,z\le\wih L(\hw)\,z
\end{equation}
for all $\hw\in\mA$ and $z\in C_2$ with $z\ge 0$.
\begin{teor}\label{3.teorcomp2}
Suppose that Hypotheses~\ref{3.hipos} hold.
\begin{itemize}
\item[(i)] If $\mM\subset\W\times\WW_2$ is a $\zeta\,$-minimal set, then
$\lambda_\mM\le\bar\lambda_\mM\le\wih\lambda_\mM$.
\item[(ii)] If $\wih\lambda_\mM<0$ for
all the $\zeta\,$-minimal sets $\mM\subset\W\times\WW_2$, then
all the conclusions of Theorem \ref{3.teorAcopy} hold.
\end{itemize}
\end{teor}
\begin{proof}
Let us fix a $\zeta\,$-minimal set $\mM\subset\W\times\WW_2$. As said before, the
semiflow $\overline\zeta_L$ is monotone on $\mM\times C_2$.
That is, if
$\hw\in\mM$ and $v_1,v_2\in C_2$ satisfy $v_1\le v_2$, then
$\bar{z}(t,\hw,v_1)\le\bar{z}(t,\hw,v_2)$
for all $t\ge 0$. Consequently,
if $\hw\in\mM$ and $v\in C_2$ then
\[
 -\bar{z}(t,\hw,|v|_m)=\bar{z}(t,\hw,-|v|_m)\le\bar{z}(t,\hw,v)\le\bar{z}(t,\hw,|v|_m),
\]
for all $t\ge 0$, so that $\n{\bar{z}(t,\hw,v)}_{C_2}\le\n{\bar{z}(t,\hw,|v|_m)}_{C_2}$.
Therefore,
\[
\begin{split}
 \bar{\lambda}_s^+(\hw,v)&:=\limsup_{t\to\infty}\frac1{t}\ln\n{\bar{z}_t(\hw,v)}_{C_2}\\
 &\le\limsup_{t\to\infty}\frac1{t}\ln
 \n{\bar{z}_t(\hw,|v|_m)}_{C_2}=\bar{\lambda}_s^+(\hw,|v|_m).
\end{split}
\]
On the other hand, \eqref{3.desigL} allows us to apply
again Theorem 1.1 of Chapter 5 of \cite{smit} in order to obtain
$0\le\bar{z}(t,\hw,|v|_m)\le\wih{z}(t,\hw,|v|_m)$
for all $t\ge 0$, which ensures that
$\n{\bar{z}(t,\hw,|v|_m)}_{C_2}\le\n{\wih{z}(t,\hw,|v|_m)}_{C_2}$.
In turn, this inequality yields
\[
 \bar{\lambda}_s^+(\hw,|v|_m)
 \le\limsup_{t\to\infty}\frac1{t}
 \ln\n{\wih{z}_t(\hw,|v|_m)}_{C_2}=:\wih\lambda^+_s(\hw,|v|_m)\,.
\]
Altogether, we have
$\bar{\lambda}_s^+(\hw,v)\le\bar{\lambda}_s^+(\hw,|v|_m)\le
\wih\lambda^+_s(\hw,|v|_m)\le\wih\lambda_\mM$ for all $\hw\in\mM$ and $v\in C_2$.
Consequently, $\bar{\lambda}_{\mM}\le\wih{\lambda}_{\mM}$.
Applying Theorem~\ref{3.teorcomp1}(ii), we have $\lambda_{\mM}\le\bar\lambda_\mM
\le\wih{\lambda}_{\mM}$.
This proves (i). Statement (ii)
is an immediate consequence of (i) and Theorems \ref{3.teor52} and \ref{3.teorAcopy}.
\end{proof}
\begin{nota}\label{3.notaventajas}
1.~Regarding the upper Lyapunov exponent of a minimal subset of
$\W\times\WW_2$, one of the advantages
of working with cooperative families (as is the case of
\eqref{3.eqvarcomp} and \eqref{3.eqvarmay}) is that the exponent
can be obtained by computing \eqref{2.lyapind} for any solution
of a particular one of the systems corresponding to
a strongly positive initial state;
that is, by a vector $v\in C_2$ such that
$v_i(s)>0$ for all $s\in[-r,0]$ and $i=1,2$. To check this assertion, we use the
notation corresponding to the flow $\overline\zeta_L$, and
note that: (1) $\bar\lambda_s(\hw,v)\le
\bar\lambda_s(\hw,|v|_m)$ (see the proof of Theorem~\ref{3.teorcomp2}(i));
(2) $|v|_m=v$ if $v$ is positive; and (3) for all pair of strongly positive
vectors $v$ and $w$ in $C_2$, there exists a constant $\alpha>0$ such that
$(1/\alpha)\,w\le v\le \alpha w$. And we also use the ergodic uniqueness of
the base flow, to ensure the independence with respect to the chosen system.
\par
2.~Note also that, in the case of \eqref{3.eqvarmay}, the coefficients
$\wih A_i$, $\wih B_i$, and $\wih C_i$ are
evaluated along the semiorbits of the base flow
$(\W,\sigma,\R)$ instead of those of the flow over $\mA$ (which are, in general, unknown).
Although the equations still depend on the semiorbits on $\mA$ through the delay,
there are many results on exponential stability which are given in terms of
$\wih A_i$, $\wih B_i$, and $\wih C_i$ and independent of the delay.
Basically, they consist
in comparing the negativeness of $\wih A_i$ with the positiveness of
$\wih B_i$ and $\wih C_i$: see e.g.~\cite{hale2}, \cite{drzo}, \cite{wu}, \cite{aros},
and references therein. An example of this situation
is given in the proof of Proposition \ref{4.propcopy}.
\end{nota}
\section{Numerical experiments}\label{sec4}
In this section we perform numerical experiments on the
family of SDDEs
\begin{equation}\label{4.main}
\left \{\;
\begin{aligned}
\dot y_1(t)&=(-1+0.3\cos(\theta_1+\sqrt{2}\,t ))\,y_1(t)\\
    &\quad\;+(-1.5+\sin(\theta_2+t)) \tanh(y_1(t))  \\
    &\quad\;+(1.4-\sin(\theta_2+t))\tanh(y_2(t-\tau(y_1,y_2)))+0.4\,,\\[.1cm]
\dot y_2(t)&=(-1.2+0.3\sin(\theta_1+\sqrt{2}\,t))\,y_2(t)\\
    &\quad\;+(1.3-0.3 \cos(\theta_2+t))\tanh(y_1(t-\tau(y_1,y_2)))\\
    &\quad\;+(-1+0.3\cos(\theta_2+t))\tanh(y_2(t))+0.4
\end{aligned}
\right.
\end{equation}
for $t\ge0$, where $(\theta_1,\theta_2)\in\T^2$: for this example $\W$
is the 2-torus, which we identify with
$(\R/[0,2\pi])^2$, and the (Kronecker) flow $\sigma$ is given by
$\sigma(t,(\theta_1,\theta_2))\equiv (\theta_1,\theta_2){\cdot}t:=(\theta_1+\sqrt{2}t,\theta_2+\,t)\in\T^2$.
The family \eqref{4.main} is obtained via the hull procedure from
the single SDDE \eqref{4.main}$_{(0,0)}$. Therefore, it
satisfies Hypotheses \ref{3.hipos} whenever the delay
$\tau\colon\R^2\to[0, r]$ is $C^1$, and the results of Theorem
\ref{3.teoratt} apply: the skew product semiflow
$(\T^2\times\WW_2,\zeta\,,\R^+)$ induced by the solutions of \eqref{4.main}
has a connected global attractor $\mA$.
\par
Our first purpose is to estimate a subset of $\T^2\times \WW_2$ as
small as possible in which the attractor $\mA$ lies.
This procedure, completed in
Corollary \ref{4.propzona},
will not depend on the specific expression of the delay.
But later we will take
\[
 \tau^\delta(y_1,y_2):=1-\sin(\delta\,(y_1+y_2))
\]
in order to prove, in Proposition~\ref{4.propcopy},
that the attractor is a copy of the base.
\par
Given $x_0\in\WW_2$, we represent by
$y(t,\theta_1,\theta_2,x_0)$ the solution of
\eqref{4.main}$_{(\theta_1,\theta_2)}$
with $y(s,\theta_1,\theta_2,x_0)=x_0(s)$ for $s\in[-r,0]$.
\subsection{A set containing the global attractor}\label{4.secRectangle}
Let us begin with the estimation of a \lq\lq small\rq\rq~subset of $\T^2\times \WW_2$
containing the attractor $\mA$.
To this end we define the functions
$F_i\colon\T^2\times\R^2\to\R^2$ for $i=1,2$
by
\[
\begin{split}
F_1(\theta_1,\theta_2,k_1,k_2)&:=(-1+0.3\cos(\theta_1))\,k_1+
(-1.5+\sin(\theta_2))\tanh(k_1)\\
&\quad\;\,+(1.4-\sin(\theta_2))\tanh(k_2)+0.4\,,\\[.1cm]
F_2(\theta_1,\theta_2,k_1,k_2)&:=(-1.2+0.3\sin(\theta_1))k_2+
(1.3-0.3 \cos(\theta_2))\tanh(k_1)\\
&\quad\;\,+(-1+0.3\cos(\theta_2))\tanh(k_2)+0.4\,,
\end{split}
\]
so that \eqref{4.main}$_{(\theta_1,\theta_2)}$ agrees with
\[
\left \{\;
\begin{aligned}
\dot y_1(t)&=F_1((\theta_1,\theta_2){\cdot}t,y_1(t),y_2(t-\tau(y_1,y_2)))\,, \\[.05cm]
\dot y_2(t)&=F_2((\theta_1,\theta_2){\cdot}t,y_1(t-\tau(y_1,y_2)),y_2(t))
\end{aligned}\right.
\]
for $t\ge0$. Below we will find sequences $(k_{1n}^\pm)$ and
$(k_{2n}^\pm)$ such that
\begin{equation}\label{4.ineqk}
 k_{in}^-<k_{i,n+1}^-<k_{i,n+1}^+<k_{i,n}^+
 \quad\text{for $i<1,2\;$ and $\;n\ge 0$}\,,
\end{equation}
and such that if
$\mR_n:=[k_{1,n}^-,k_{1,n}^+] \times [k_{2,n}^-,k_{2,n}^+]\subset\R^2$
for $n\ge 0$, then
\begin{prop}\label{4.propRn}
$\mA\subseteq\T^2\times\wit\mR_n$ for
$\wit\mR_n:=\{\varphi\in\WW_2\,|\;\varphi\colon[-r,0]\to\mR_n\}$.
\end{prop}
For ease of notation we drop the explicit dependence of $F_{i}$
on $(\theta_1,\theta_2)$. We build the four sequences
$(k_{1n}^\pm)_{n\ge 1}$ and $(k_{2n}^\pm)_{n\ge 0}$
following the next steps:
\begin{list}{}{\leftmargin 22pt}
\item[{\bf S$_{0\text{ }}$}]
    Initialize the sequences $(k_{1n}^\pm)_{n\ge 0}$ and $(k_{2n}^\pm)_{n\ge 0}$ to the following values:
    \[
     k_{i0}^+:=+\infty \qquad\text{and}\qquad k_{i0}^-:=-\infty
     \qquad \text{for $\,i=1,2$}\,.
    \]
\item[{\bf S$_{11}^+$}]
    We set $F_{11}^+(k):=F_1(k,k_{20}^+)$ and note that
    $F_1(k,k_2)\le F_{11}^+(k)$ for all $k_2\le k_{20}^+$
    and all $(\theta_1,\theta_2)\in\T^2$.
    We want to find a $k_{11}^+$ such that $F_{11}^+(k)<0$
    for all $(\theta_1,\theta_2)\in\T^2$ whenever $k>k_{11}^+$.
    Note that $F_{11}^+(0)=(1.4-\sin(\theta_2))+0.4>0$ and hence
    $k_{11}^+$ must necessarily be greater than $0$. For $k>0$ and
    all $(\theta_1,\theta_2)\in\T^2$,
    $F_{11}^+(k)\le G_{11}^+(k)$, where
    \[
    \begin{split}
    G_{11}^+(k)&:=-0.7\,k-1.5 \tanh(k)+1.4\tanh(k_{20}^+)\\    &\;\quad\quad\quad+|\tanh(k)-\tanh(k_{20}^+)\,|+0.4\,.
    \end{split}
    \]
    This is a strictly decreasing function, so that we achieve our goal by
    taking $k_{11}^+$ as the unique point with $G_{11}^+(k_{11}^+)=0$.
\item[{\bf S$_{11}^-$}] We set $F_{11}^-(k):=F_1(k,k_{20}^-)$ and observe that  $F_1(k,k_2)>F_{11}^-(k)$ for all $k_2\ge k_{20}^-$
    and all $(\theta_1,\theta_2)\in\T^2$.
    We want to find a $k_{11}^-$ such that $F_{11}^-(k)>0$
    for all $(\theta_1,\theta_2)\in\T^2$ whenever $k<k_{11}^-$.
    Let us define the continuous and strictly decreasing function
    \[
     G_{11}^-(k):=\left\{\!\!\begin{array}{l}
     -0.7\,k-1.5\tanh(k)+1.4\tanh(k_{20}^-))\\[.05cm]
     \qquad-|\tanh(k)-\tanh(k_{20}^-)\,|+0.4
     \qquad\qquad\qquad \text{ if $k<0$}\,, \\[.1cm]
     -1.3\,k-1.5\tanh(k)+1.4\tanh(k_{20}^-))\\[.05cm]
     \qquad-|\tanh(k)-\tanh(k_{20}^-)\,|+0.4
     \qquad\qquad\qquad\text{ if $k\ge0$}\,.
    \end{array} \right.
    \]
    Then $F_{11}^-(k)\ge G_{11}^-(k)$ for all $k\in\R$ and
    all $(\theta_1,\theta_2)\in\T^2$. We take $k_{11}^-$ such that
    $G_{11}^-(k_{11}^-)=0$.
\item[{\bf S$_{21}^+$}] We set $F_{21}^+(k):=F_2(k_{10}^+,k)$ and note that
    $F_2(k_1,k)\le F_{21}^+(k)$ for all $k_1\le k_{10}^+$
    and all $(\theta_1,\theta_2)\in\T^2$.
    As in {\bf S$_{11}^+$}, we look for $k_{21}^+$ (which must be greater than $0$)
    such that $F_{21}^+(k)<0$
    for all $(\theta_1,\theta_2)\in\T^2$ whenever $k>k_{21}^+$,
    what we achieve by taking
    $k_{21}^+$ as the unique zero of the strictly decreasing function
    \[
    \begin{split}
     G_{21}^+(k)&:=-0.9\,k-\tanh(k)+1.3\tanh(k_{10}^+)\\
     &\quad\quad\quad\quad +0.3\,|\tanh(k)-\tanh(k_{10}^+)\,|+0.4\,.
    \end{split}
    \]
\item[{\bf S$_{21}^-$}] We set $F_{21}^-(k):=F_2(k_{10}^-,k)$ and note that
    $F_2(k,k_2)\ge F_{21}^-(k)$ for all $k\ge k_{10}^-$
    and all $(\theta_1,\theta_2)\in\T^2$.
    Reasoning as in {\bf S${11}^-$} we define $k_{21}^-$ as the unique
    zero of the continuous and strictly decreasing function
    \[
     G_{21}^-:=\left\{\!\!\begin{array}{l}
     -0.9\,k-\tanh(k)+1.3\tanh(k_{10}^-)\\[.05cm]
     \qquad -0.3\,|\tanh(k)-\tanh(k_{10}^-)\,|+0.4
     \quad\qquad\qquad \text{ if $k<0$}\,, \\[.1cm]
     -1.5k-\tanh(k)+1.3\tanh(k_{10}^-)\\[.05cm]
     \qquad -0.3\,|\tanh(k)-\tanh(k_{10}^-)\,|+0.4
     \quad\qquad\qquad \text{ if $k\ge 0$}\,.
     \end{array}\right.
    \]
\end{list}
In the steps {\bf S$_{1n}^\pm$} we define $k_{1n}^\pm$
by repeating steps {\bf S$_{11}^\pm$}
with $k_{2,n-1}^\pm$ in place of
$k_{20}^\pm$; and in the steps {\bf S$_{2n}^\pm$} we define $k_{2n}^\pm$
by repeating steps {\bf S$_{21}^\pm$}
with $k_{1,n-1}^\pm$ in place of $k_{10}^\pm$.
It is not hard (by applying an induction procedure
based on comparing the functions $G_{in}^\pm$, whose zeroes
determine the four sequences) that \eqref{4.ineqk} hold.
\medskip\par\noindent
{\em Proof of Proposition \ref{4.propRn}}.~Is is obvious that
$\mA\subset\wit\mR_0=\R^2$. We will assume that
$\mA\subset\wit\mR_{n-1}$ and prove that
$\mA\subset\wit\mR_n$. Recall that
$\mA$ is composed by the globally defined and bounded orbits:
see Remarks~\ref{2.notas}.3 and \ref{2.notapseudo}.1.
Let us assume that there exists $(\bar\theta_1,\bar\theta_2,\bar x)\in\mA$ such that
$\bar x\notin\wit\mR_n$. There is no restriction in assuming that
$\bar x(0)\notin\mR_n$. We will get a contradiction supposing that
$\bar k_1=\bar x_1(0)>k_{1n}^+$,
since the remaining possibilities are analyzed in an analogous way.
We define $\mK_2:=\{x_2(s)\,|\;(\theta_1,\theta_2,x)\in\mA\,,\;s\in[-r,0]\}$,
which is a compact subset of $[k^-_{2,n-1},k^+_{2,n-1}]$, and
$\delta:=\max_{(\theta_1,\theta_2)\in\T^2,k_2\in\mK_2}F_1(\theta_1,\theta_2,
k_{1n}^+,k_2)$. Observe that the expression of $F_1$ and the inequalities in
{\bf S$_{1n}^+$} ensure that $F_1(\theta_1,\theta_2,k_1,k_2)<
F_1(\theta_1,\theta_2,k_{1n}^+,k_2)\le\delta<0$ whenever
$k_1>k_{1n}^+$, $k_2\in\mK_2$ and $(\theta_1,\theta_2)\in\T^2$.
\par
Let us represent by $\{((\bar\theta_1,\bar\theta_2){\cdot}t,\bar u(t))\,|\;t\in\R\}\subset\mA$ a
complete orbit for $(\bar\theta_1,\bar\theta_2,\bar x)$
(which is its value at $0$),
so that: $\bar y(t):=\bar u(t)(0)$
satisfies \eqref{4.main}$_{(\bar\theta_1,\bar\theta_2)}$
for all $t\in\R$, and the second
component $\bar y_2$ takes values in $\mK_2$. Then
$\dot{{\bar y}}_1(0)\le\delta$, and it is easy to deduce
by contradiction that
$\dot{{\bar y}}_1(t)\le\delta$ for all $t<0$. But this precludes
the boundedness of $\bar y(t)$, which is the sought-for contradiction.
\qed\medskip
\par
The fact that Proposition \ref{4.propRn} is valid for $n\ge 0$
proves the next assertion.
\begin{coro}\label{4.propzona}
Let us define $\wit\mR:=\bigcap_{\,n\ge 0}\mR_n$. Then
$\mA\subseteq\T^2\times\wit\mR$\,.
\end{coro}
Note also that, if
$k_i^\pm:=\lim_{n \to \infty} k_{in}^\pm$ for $i=1,2$ and
$\mR:=[k_1^-,k_1^+] \times [k_2^-,k_2^+]\subset\R^2$
then $\wit\mR=\{\varphi\in\WW_2\,|\;\varphi\colon[-r,0]\to\mR\}$.
We also remark that, for $i=1,2$, the functions $G_{in}^\pm$
do not depend on the delay, and
hence the area containing the global attractor is independent
of the particular choice of $\tau$:
for all the families \eqref{4.main}
we obtain as estimate of the rectangle $\mR=[\,0.2910,\,0.5705\,]\times[\,0.3090,\,0.5999\,]$.
The sequences of zeros are obtained using a standard
Matlab command, working with a tolerance of $10^{-5}$.
We will make use of the less accurate bounds
\begin{equation}\label{4.R}
 \mR\subset[\,0.28,\,0.58\,]\times[\,0.30,\,0.61\,]
\end{equation}
in order to prove the next result:
\begin{prop}\label{4.propcopy}
Let us take the delay $\tau^\delta(y_1,y_2):=1-\sin(\delta\,(y_1+y_2))$.
Then, if $|\delta|$ is small enough,
the upper Lyapunov exponent of the corresponding global attractor
$\mA^\delta$ is strictly negative,
and consequently $\mA^\delta$ is a copy of $\,\T^2$.
\end{prop}
\begin{proof}
Note that the second assertion follows from the first one and
Theorem \ref{3.teorAcopy}. Let us consider the family
\eqref{3.eqvarmay}$_\delta$ constructed from \eqref{4.main} for
the delay $\tau_\delta$, given by
\[
\begin{split}
\wih A_1^\delta(\theta_1,\theta_2)&:=-1+0.3\cos\theta_1 -
\ep_{1f}^\delta(1.5-\sin\theta_2)-\ep_{2g}^\delta(1.4-\sin\theta_2)\,,\\
\wih B_1^\delta(\theta_1,\theta_2)&:=\kappa^\delta_{3g}(1.4-\sin\theta_2)\,,\\
\wih C_1^\delta(\theta_1,\theta_2)&:=\kappa^\delta_{4g}(1.4-\sin\theta_2)\,,\\
\wih A_2^\delta(\theta_1,\theta_2)&:=-1.2+0.3\sin\theta_1-
\ep_{1g}^\delta(1-0.3\cos\theta_2)-\ep_{2f}^\delta(1.3-0.3\cos\theta_2)\,,\\
\wih B_2^\delta(\theta_1,\theta_2)&:=\kappa_{3f}^\delta(1.3-0.3\cos\theta_2)\,,\\
\wih C_2^\delta(\theta_1,\theta_2)&:=\kappa_{4f}^\delta(1.3-0.3\cos\theta_2)\,,
\end{split}
\]
where the constants are defined by the expressions given before
Theorem \ref{3.teorcomp2}, now corresponding to $\tau^\delta$ and
$\mA^\delta$. According to Proposition 5.3 of
of \cite{noos4}, it is enough to prove that
\begin{equation}\label{4.ly}
 \wih A_1^\delta+\wih B_1^\delta+\wih C_1^\delta<0 \qquad\text{and}\qquad
 \wih A_2^\delta+\wih B_2^\delta+\wih C_2^\delta<0\,.
\end{equation}
Note that $f$ and $g$ are given in our case by $\tan h$, whose
derivative decreases in $\R^+$, and that Corollary~\ref{4.propzona} and
\eqref{4.R}
yield $x(s)\in[\,0.28,\,0.58\,]\times[\,0.30,\,0.61\,]$
for all $s\in[-2,0]$ and $(\theta_1,\theta_2,x)\in\mA^{\delta}$.
This ensures that
$\ep^\delta_{1f}\ge\dot f(0.58)$, $\ep^\delta_{1g}\ge \dot g(0.61)$,
$\kappa^\delta_{4g}\le \dot g(0.30)$ and $\kappa^\delta_{4f}\le \dot f(0.28)$.
Therefore,
\[
\begin{split}
 \wih A_1^\delta+\wih B_1^\delta+\wih C_1^\delta &\le
 -0.7-2.5\,\dot f(0.58)
 +2.4\,\dot g(0.30)+2.4\big(|\kappa^\delta_{3g}|+|\ep_{2g}^\delta|\big)\\
 &< -0.3+2.4\big(|\kappa^\delta_{3g}|+|\ep_{2g}^\delta|\big)\,,\\
 \wih A_2^\delta+\wih B_2^\delta+\wih C_2^\delta &\le
 -0.9-1.3\,\dot g(0.61)+
 1.6\,\dot f(0.28)+1.6\big(|\kappa^\delta_{3f}|+|\ep_{2f}^\delta|\big)\\
 &< -0.3 +1.6\big(|\kappa^\delta_{3f}|+|\ep_{2f}^\delta|\big)\,.
\end{split}
\]
It is also easy to deduce from the compactness of $\mA$ and from the
expression of $\tau^\delta$ that there exists $c>0$ such that
$|\kappa^\delta_{3g}|$, $|\ep_{2g}^\delta|$, $|\kappa^\delta_{3f}|$
and $|\ep_{2f}^\delta|$ are bounded by $c\,|\delta|$, which shows that
\eqref{4.ly} holds if $|\delta|$ is small enough.
\end{proof}
\subsection{Numerical simulations}\label{Numerical_sec}
In what follows we report on numerical simulations of \eqref{4.main}
with delay $\tau(y_1,y_2):=1-\sin(0.1\,(y_1+y_2))$.
All our computations are done with the Matlab function {\tt ddesd}.
When not specified, we use the default options. We simulate the dynamic,
visualize a set $\mA$ that attracts all numerical solutions and give
numerical evidence that $\mA$ is a one copy of the base (so that
$\delta=0.1$ is one of the values provided by Proposition~\ref{4.propcopy}).
\par
For convenience, we choose a finite set
$X\subset W^{1,\infty}_2$ of initial conditions with constant value
in $[-2,0]$ and we solve \eqref{4.main} for each
$x_0\in X$ and for a $(\theta_1^0,\theta_2^0)$ fixed:
$(0,0)$ for Figure \ref{attractor_fig}. Let us
call $y(t,x_0):=y(t,0,0,x_0)$.
In Figure \ref{attractor_fig} on the left we plot in the
$y_1-y_2$ plane the solutions $y(t,x_0)$
of \eqref{4.main}$_{(0,0)}$ for $x_0\in X$ computed
up to time $T=2000$. On the right, for the same set $X$,
we plot $(y_1(t,x_0),y_2(t,x_0),\theta_2(t))$ for $t\in[1000, \,\ 2000]$,
after getting rid of transient.
\begin{figure}[h]
\caption{Simulations of \eqref{4.main}$_{(0,0)}$.
Left: different solutions converging to $\mA$.
Right: plot of $(y_1(t),y_2(t),\theta_2(t))$ for
$t\in[1000,2000]$ for the same set of solutions.}\label{attractor_fig}
$\begin{matrix}
\includegraphics[width=175pt]{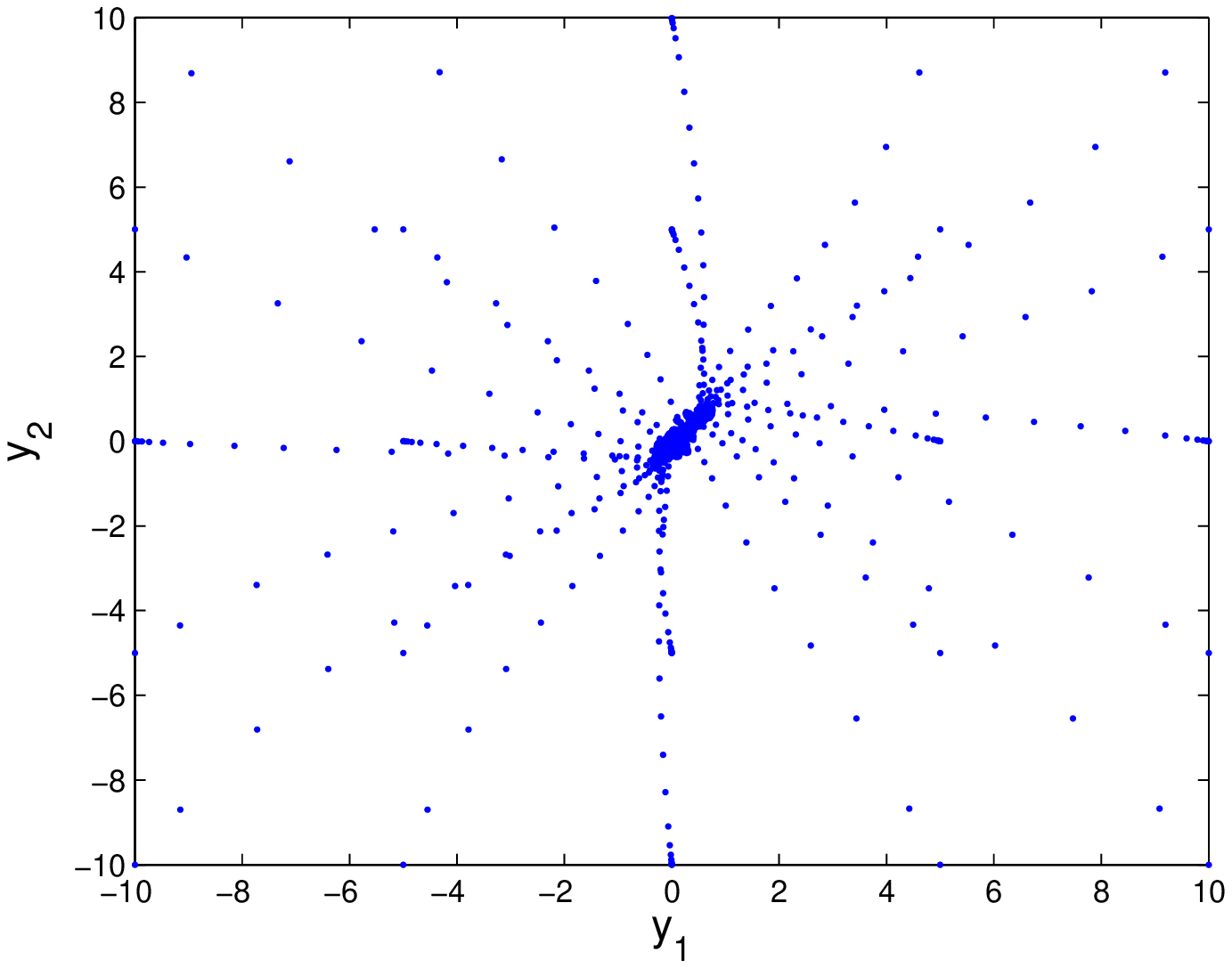} &
\includegraphics[width=175pt]{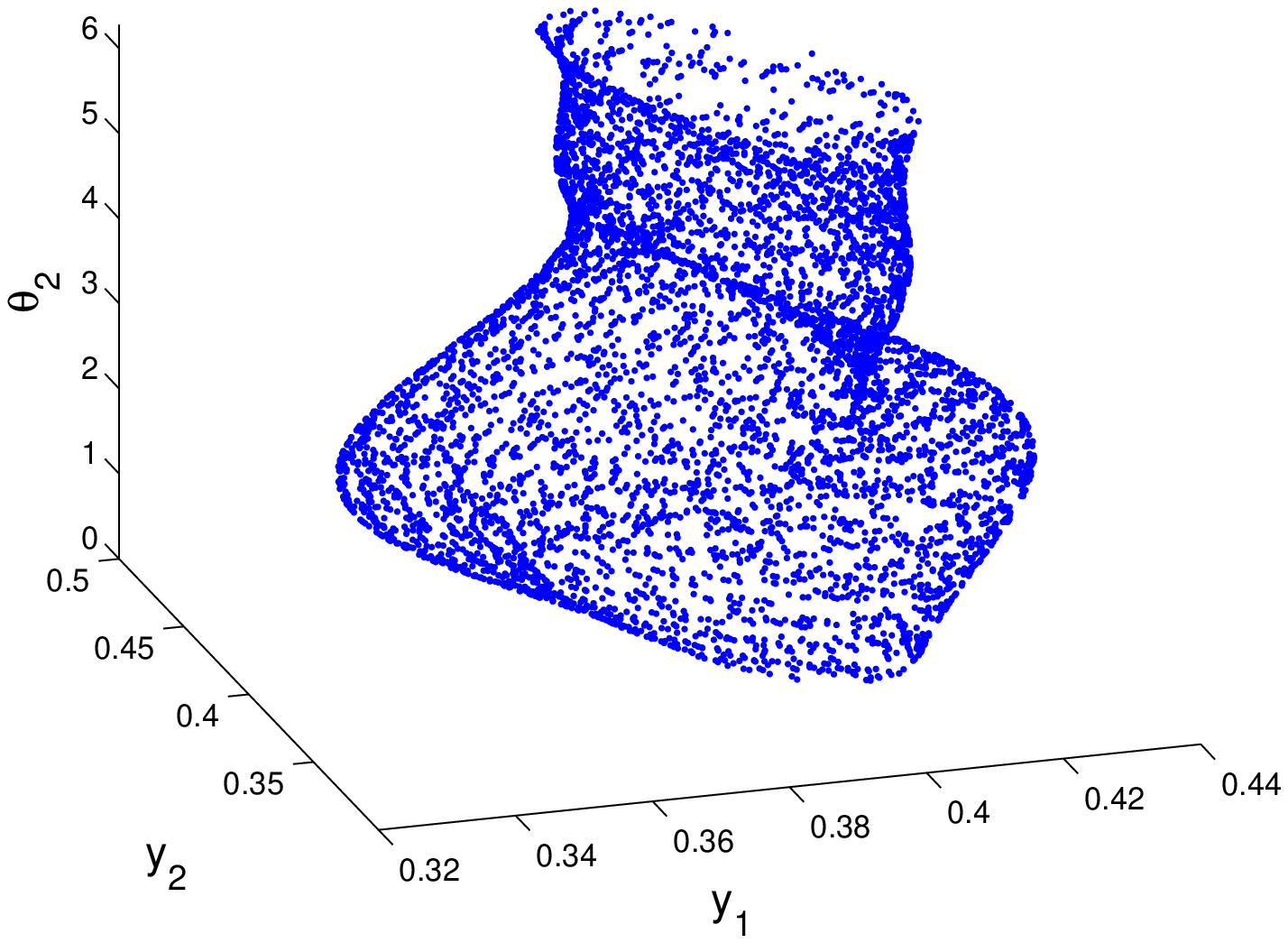}
\end{matrix}
$
\end{figure}
\par
In Figure \ref{attractor_compare_fig}, we depict the numerical attractor
for \eqref{4.main}$_{(\theta_1,\theta_2)}$ with $(\theta_1,\theta_2)\ne(0,0)$
and compare it with that obtained for $(0,0)$. On the left we plot
the attractor in the $y_1 - y_2$ plane
for $(0,0)$.
On the right, in blue we plot the numerical solutions
for $(0,0)$ and in red the numerical solutions for
$(\pi/3,\pi/2)$. Similar plots are obtained
for other systems in the family. Note that the numerical
attractors are contained in the set $\mR$ obtained in
Section \ref{4.secRectangle}.
All the above computations are performed with {\tt RelTol}$=10^{-9}$ in {\tt ddesd}.
\begin{figure}[h]
\caption{Plot of $(y_1(t),y_2(t))$ for different numerical solutions. Left: $(\theta_1^0,\theta_2^0)=(0,0)$.
Right: in blue $(\theta_1^0,\theta_2^0)=(0,0)$, in red
$(\theta_1^0,\theta_2^0)=(\pi/3,\pi/2)$.}\label{attractor_compare_fig}
$\begin{matrix}
\includegraphics[width=175pt]{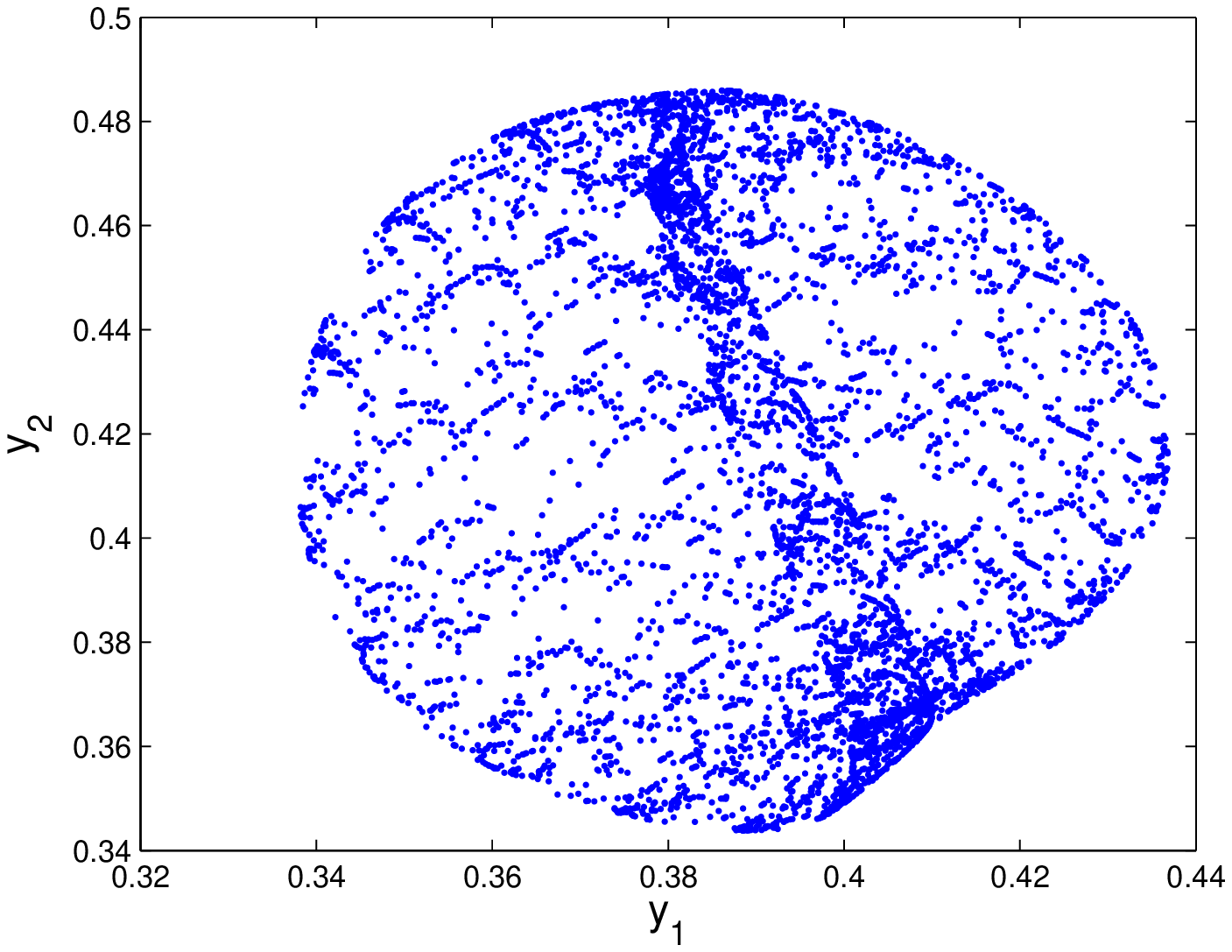} &
\includegraphics[width=175pt]{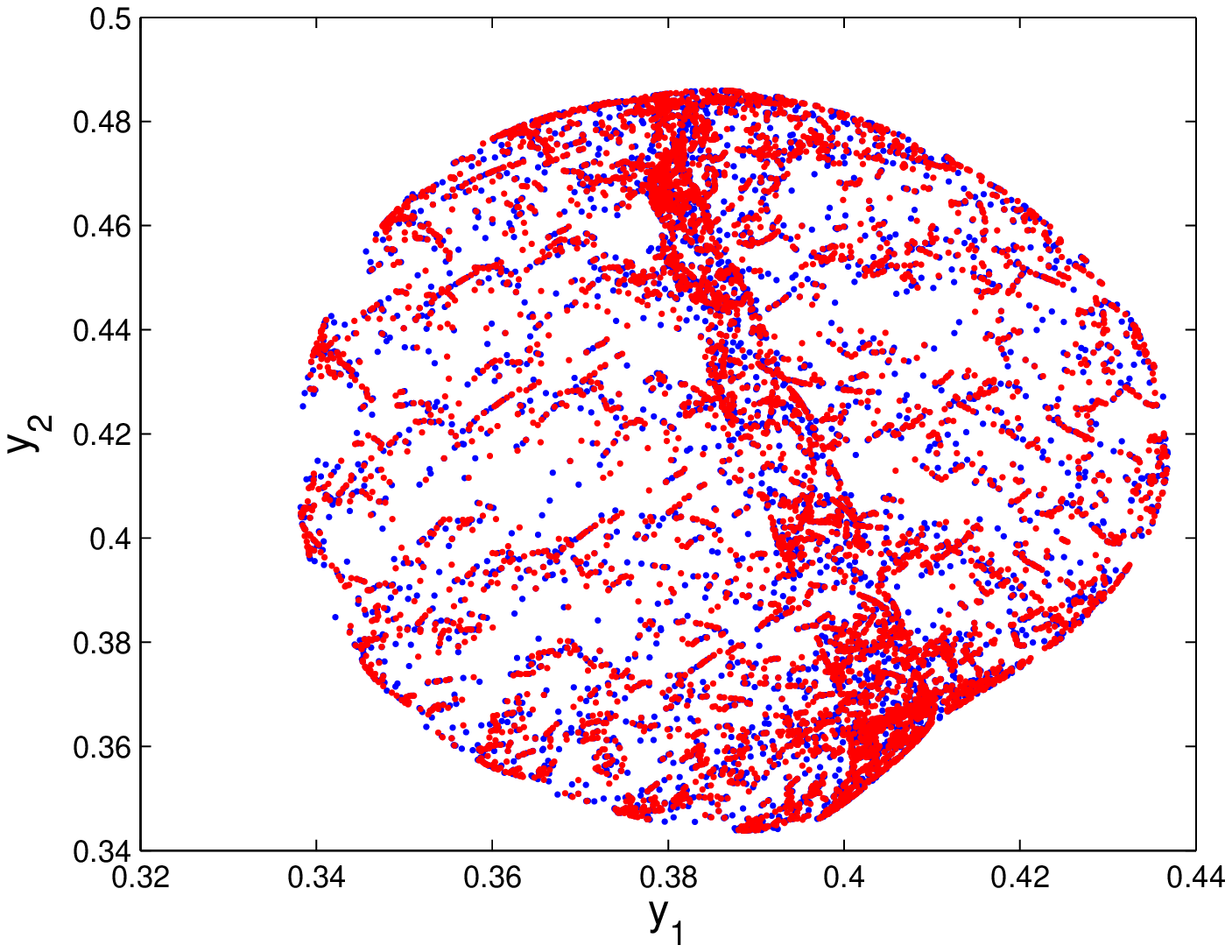}
\end{matrix} $
\end{figure}
We point out that the plots in Figure \ref{attractor_compare_fig}
do not change if,
instead of considering the whole set of initial conditions
$x_0\in X$, we consider only one initial condition $x_0$ .
This fact seems to indicate that the attractor
has negative Lyapunov exponent and hence (see Theorem~\ref{3.teorAcopy})
is a copy of the base, so that
if we fix $(\theta_1,\theta_2)$ in \eqref{4.main},
all solutions corresponding to different initial conditions will
eventually converge. This fact is also confirmed by
Figure \ref{t2_y1_t2_y2_fig},
where we plot $y_1(t,x_0)$ and $y_2(t,x_0)$ in the time interval
$[2,40]$ for all $x_0 \in X$.
\begin{figure}[h]
\caption{Plot of $(t,y_1(t,x_0))$ and $(t,y_2(t,x_0))$
for different numerical solutions.}\label{t2_y1_t2_y2_fig}
$\begin{matrix}
\includegraphics[width=175pt]{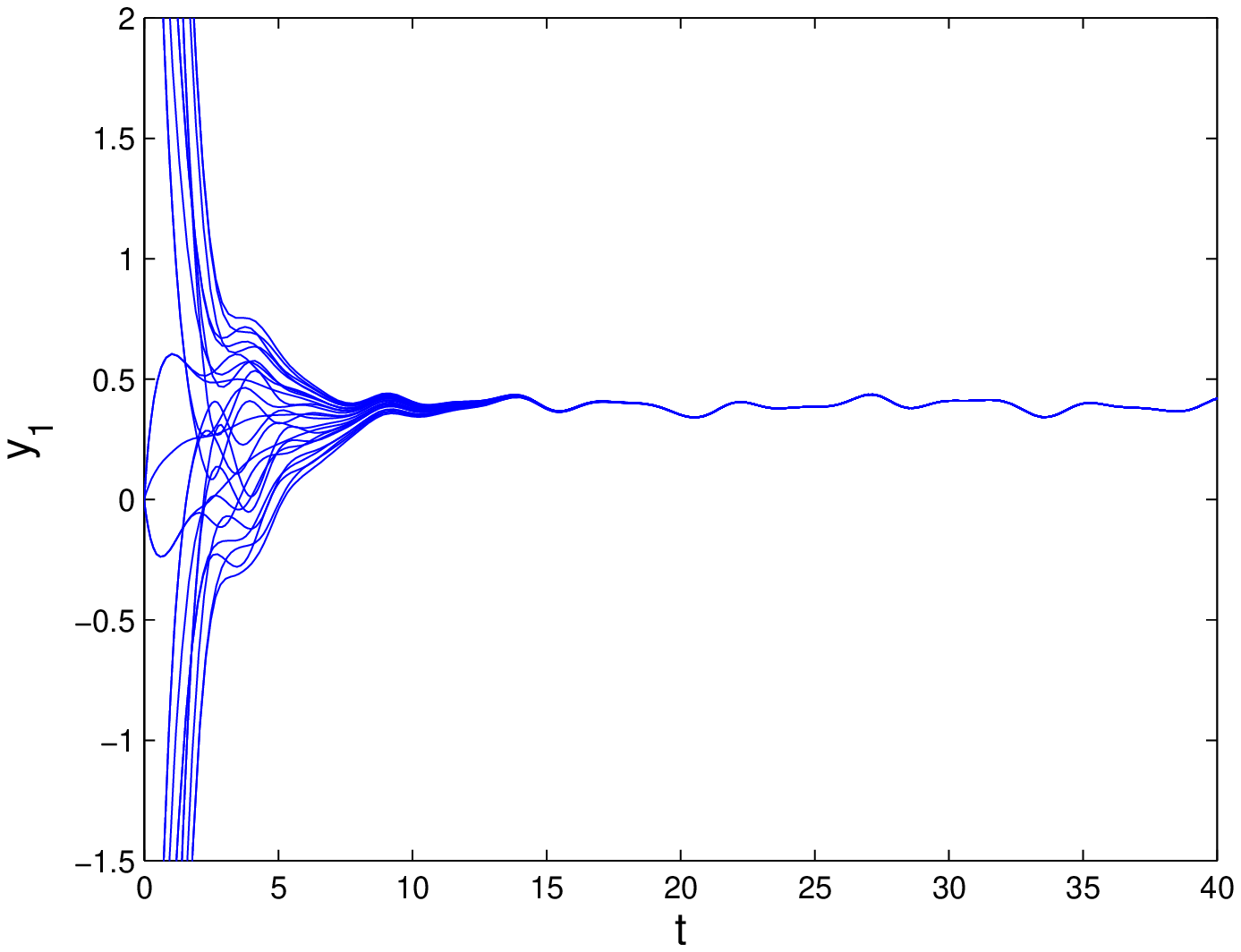} &
\includegraphics[width=175pt]{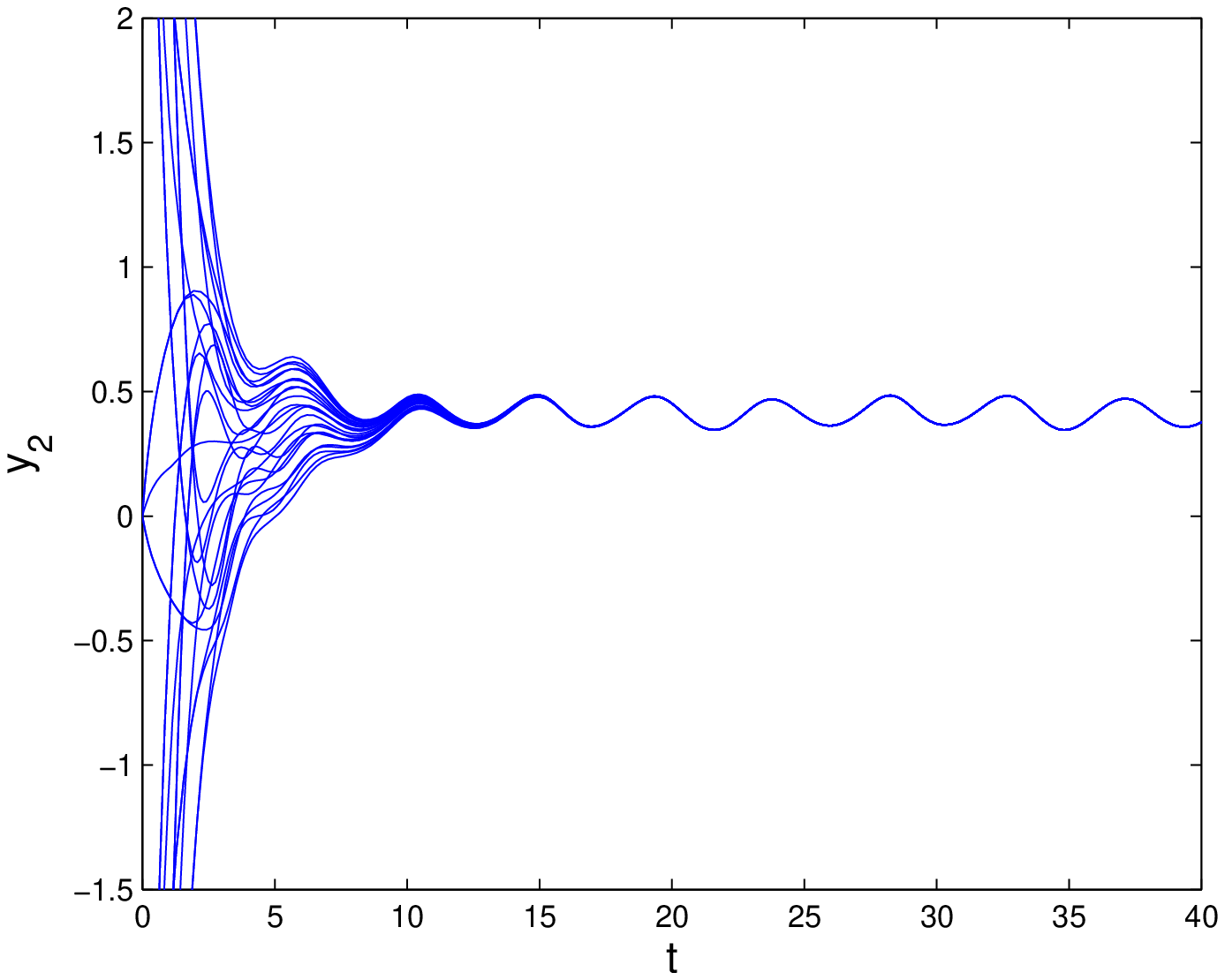}
\end{matrix} $
\end{figure}
\par
As a matter of fact, Proposition \ref{4.propcopy} guarantees that,
if $\delta=0.1$ \lq\lq is small enough\rq\rq, then
the global attractor $\mA$ is indeed a copy of the base; that is,
the graph of a continuous function $a\colon\T^2\to\wit\mR$ with
$\wit\mR$ defined in Corollary \ref{4.propzona}. But it is not clear how
to determine if $\delta=0.1$ is in that case. However, we obtain
two numerical evidences, which we describe to complete the paper.
\par
Assume that $\mA$ is the graph of $a$. Then
$a(\w)(0)=\lim_{t\to\infty} y(t,\w(-t),x_0)$ for any $x_0$,
as deduced for instance from Theorem \ref{3.teorpull}.
And, in addition, the components of the
map $\T^2\to\R^2\,,\;(\theta_1,\theta_2)\mapsto a(\theta_1,\theta_2)(0)$
define two copies of $\T^2$ in $\T^2\times\R$.
So that we first define a uniform grid of
base points $(\theta_1^j,\theta_2^k)$, $j,k=1,\ldots,20$.
We fix a tolerance $\tt tol=10^{-6}$ and an $x_0\in X$ and
compute $y^{jk}:=y(T,(\theta_1^j,\theta_2^k){\cdot}(-T),x_0)$
for $j,k=1,\ldots,20$, where $T$ is such that
the distance between the $y^{jk}$'s computed
at time $T$ and time $(T-5)$ is less than the fixed tolerance.
In Figure \ref{pullback_fig} on the left we plot
$(\theta_1^j,\theta_2^k,y_1^{jk})$ and on the right
we plot $(\theta_1^j,\theta_2^k,y_2^{jk})$.
\begin{figure}[h]
\caption{Mesh of $(\theta_1^j,\theta_2^k,y_1^{jk})$ on the left and
$(\theta_1^j,\theta_2^k,y_2^{jk})$ on the right for any
initial condition and $T=30$.}\label{pullback_fig}
$\begin{matrix}
\includegraphics[width=175pt]{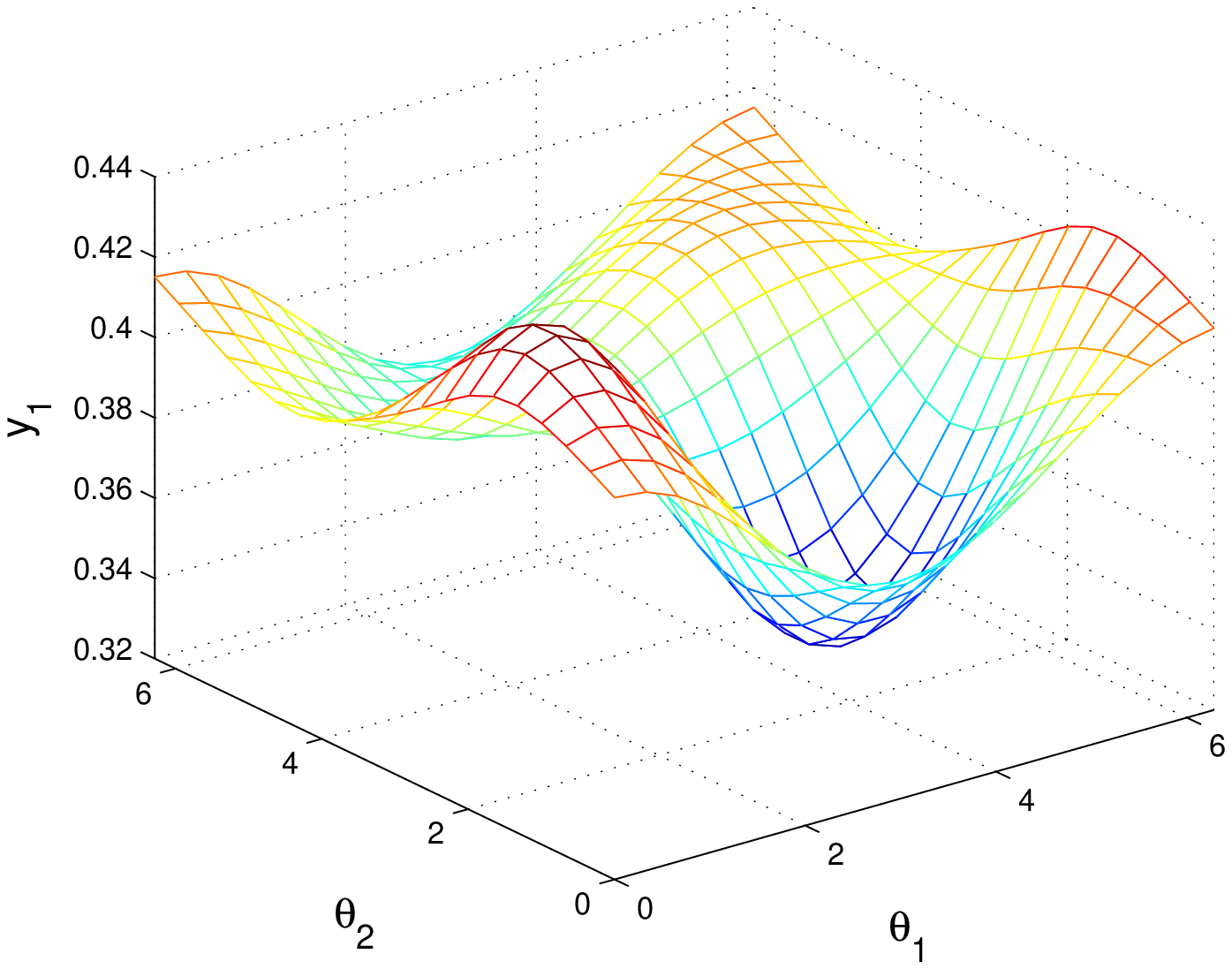} &
\includegraphics[width=175pt]{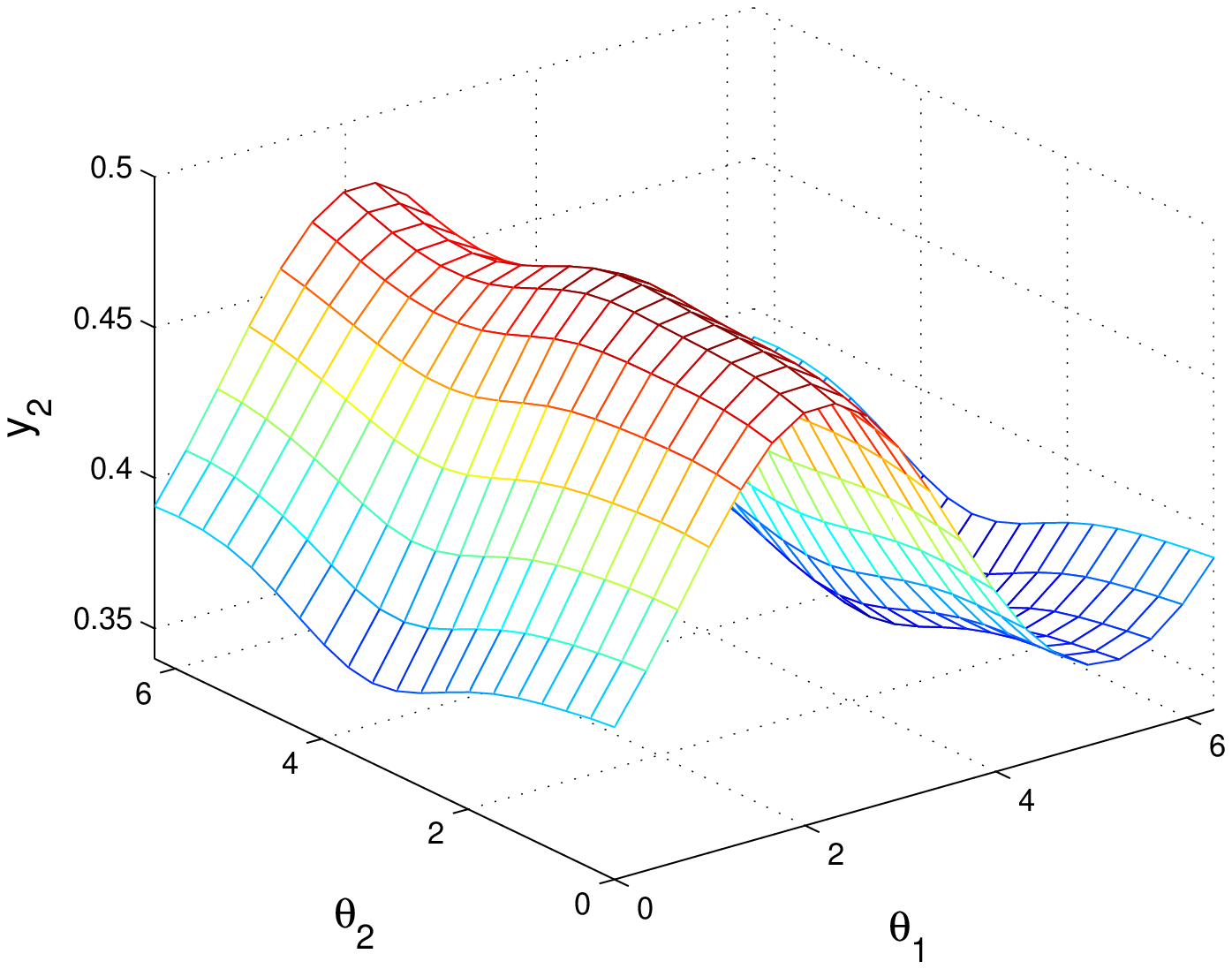}
\end{matrix}
$
\end{figure}
\par
We obtain the same plots for any other initial condition $x_0$ we consider.
Observe that the sections corresponding to $(0,\theta^k)$ and
$(2\pi,\theta^k)$ agree, likewise those for $(\theta^j,0)$ and
$(\theta^j,2\pi)$. That is, both plots seem to be graphs of
continuous functions $\T^2\to\R$, as expected.
This is the first numerical evidence of the fact that we have indeed a copy
of the base. Note also that we are acting in a \lq\lq pullback\rq\rq~sense
in order to depict the attractor.
\par
For the second evidence we will combine forward and backward methods.
First, we obtain three sections
of the plots in Figure \ref{pullback_fig}. For that,
we take a denser grid of 50 points for $\theta_1^j$, and obtain $y^j:=u(T,(\theta_1^j,\theta_2){\cdot T},x_0)(0)$
for $T=30$ and three values of $\theta_2$: $0$, $\pi/3$, $\pi/2$.
In Figure \ref{t1_y1_fig} we plot with continuous line
the three graphics for
$(\theta_1^j,y_1^j)$ on the left  and for $(\theta_1^j,y_2^j)$ on the right.
Let now act in a \lq\lq forward sense\rq\rq: we consider three
Poincar\'{e} sections for a fixed $\theta_2$, again for
$0$, $\pi/3$, $\pi/2$. The markers in Figure \ref{t1_y1_fig} correspond to
the sets of points $(\theta_1(t),y_1(t))$ (left)
and $(\theta_1(t),y_2(t))$ (right) on the three sections after
getting rid of transient.
And this time we take a different type of initial condition:
$x_0(s):=(-t,\cos(s))$ for $s\in[-2,0]$.
As expected, we obtained the same plot for any
other initial condition we considered. The fact that
the curves in Figure \ref{t1_y1_fig} overlap
constitutes the second evidence we referred to.
\begin{figure}[h]
\caption{Sections of the copies of the base of Figure
\ref{pullback_fig} for $\theta_2=0, \pi/3, \pi/2$ (continuous lines).
Plots of $(\theta_1(t),y_1(t))$, $(\theta_1(t),y_2(t))$ on the Poincar\'e
sections $\theta_2=0, \pi/3, \pi/2$ for $T\in[1000,2000]$,
with initial condition $x_0(s):=(-s,\cos(s))$ (markers).}\label{t1_y1_fig}
$\begin{matrix}
\includegraphics[width=175pt]{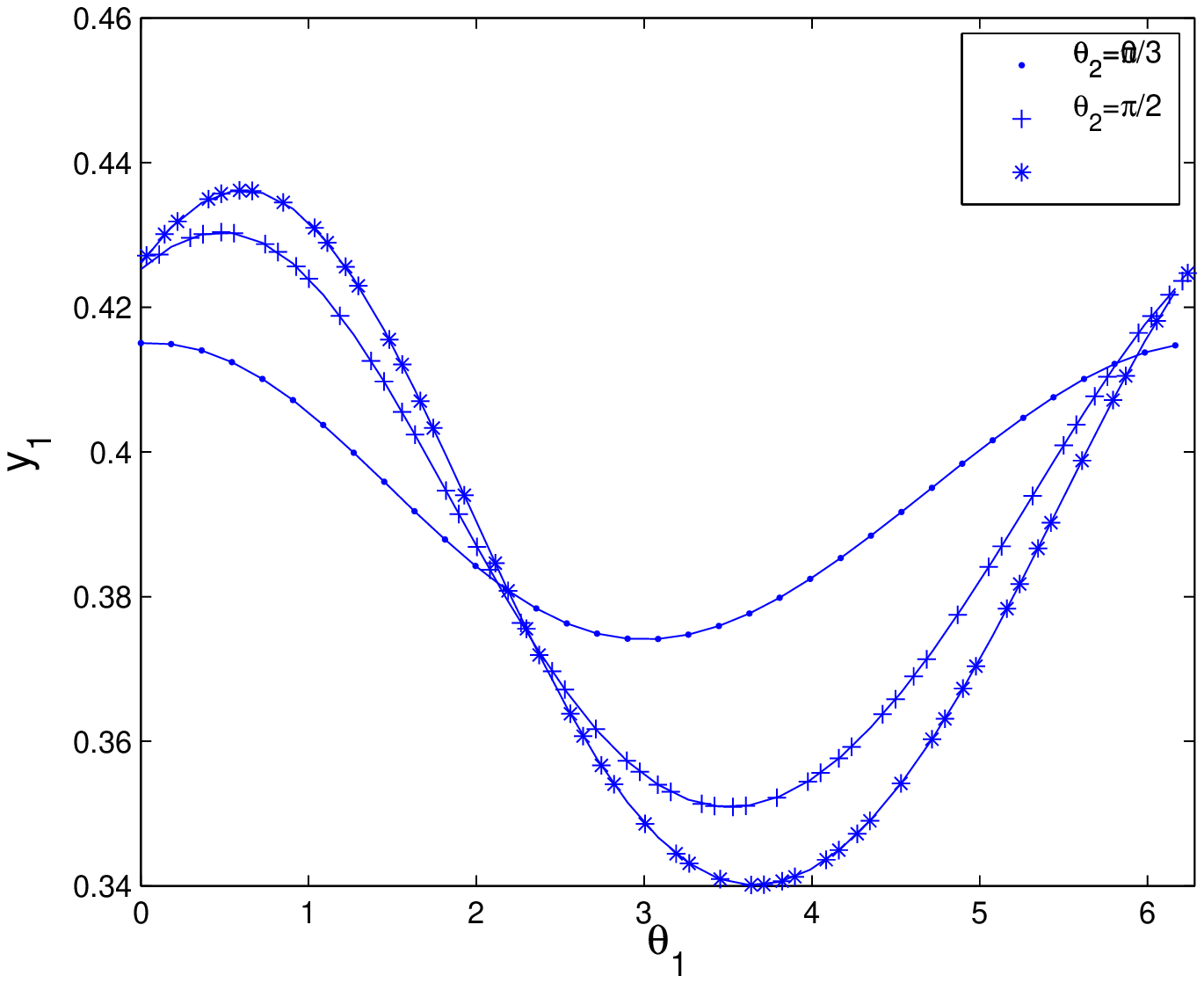} &
\includegraphics[width=175pt]{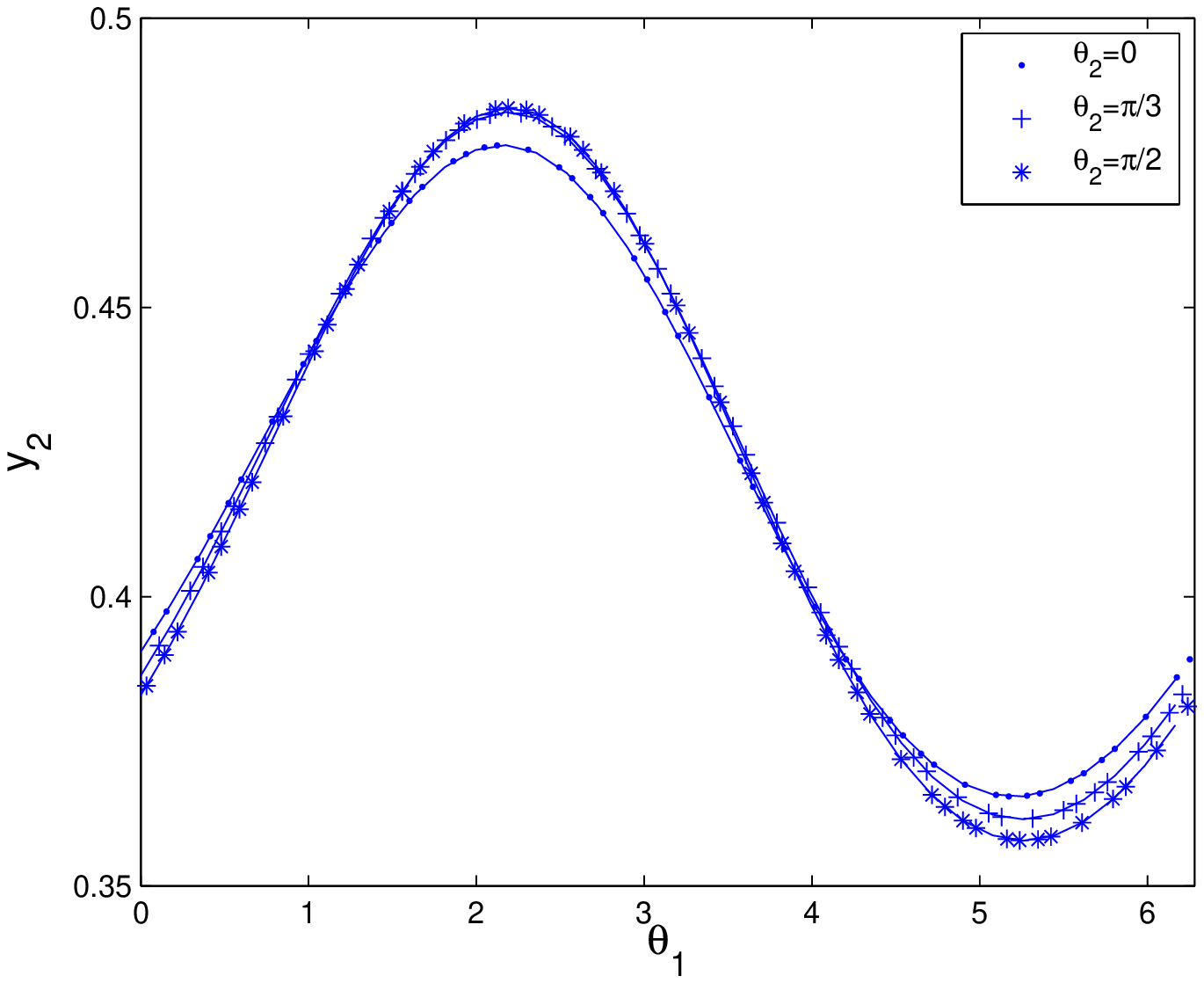}
\end{matrix}
$
\end{figure}

\end{document}